\documentclass[11pt, a4paper, openany]{amsart}
\usepackage[applemac]{inputenc}

\usepackage{t1enc}    
\usepackage{amsmath}
\usepackage[all,cmtip,2cell]{xy}
\UseTwocells
\usepackage{ucs} 
\usepackage[applemac]{inputenc}
\usepackage{textcomp}  
\usepackage{t1enc}    
\usepackage{subfiles}
\usepackage{article}
\usepackage{endnotes}
\usepackage[  hidelinks,
  bookmarks=true,
  bookmarksnumbered=true,
  pdfusetitle]{hyperref} 
\usepackage{wasysym}
\usepackage{ifsym}
\usepackage[mathcal]{eucal}
\usepackage{eufrak}
\usepackage{mathrsfs}
\usepackage{yfonts}
\usepackage{stmaryrd}
\usepackage{xcolor}
\usepackage{enumitem} 
\usepackage{lmodern}
\usepackage{slantsc}

\hypersetup{
  pdftitle    = {Alain Badiou's Mistake: Two Postulates of Dialectical Materialism},
  pdfauthor   = {Antti Veilahti},
  pdfsubject  = {Topos theory, sheaf theory, philosophy of mathematics},
  pdfkeywords = {topos theory, Badiou, forcing, sheaves, Heyting algebra},
  pdfcreator  = {LaTeX with hyperref}
}

\usepackage{fancyhdr}
\usepackage{ifthen}

\DeclareMathOperator{\Nat}{Nat}
\DeclareMathOperator{\Match}{Match}

\DeclareMathOperator{\cod}{cod}
\DeclareMathOperator{\op}{op}
\newcommand{\true}{\mathrm{true}}

\newcommand{\Si}{\mathfrak{S}}

\author{Antti Veilahti}

\title{\Large{Alain Badiou's Mistake:\\ $\ $ \\ Two Postulates of Dialectical Materialism \vspace{.5cm}}}

\begin{document}

\begin{abstract}
Originally published in 2013, this revised article constitutes the first stage of a two-part intervention that revisits Alain Badiou's deployment of topos theory in \textit{Being and Event} and \textit{Logics of Worlds} \href{https://arxiv.org/pdf/2611.27280}{(see arxiv:2611.27280)}. It interrogates the structural and philosophical consequences of taking the category-theoretic foundations of mathematics seriously, extending the boundaries of formal discourse. 

This part focuses on \textit{Logics of Worlds}, which emerged as Badiou's response to his own theory of subjectivity. Transferring its key structures like transcendental indexing and universal exposition to topos theory, I show how his concept of `dialectical materialism' remains constrained by set-theoretic ontology and localic topos theory---essentially a theory of intuitionistic logic. In particular, Badiou's claim that categorical thinking does not substantially change mathematics overlooks the true geometric momentum driving topos theory. 

While the primary contents of this article are technical, I further explore how Badiou's formal omissions critically impact his philosophical project. This is especially visible in regard to his `postulate of materialism'---a device he utilizes to promote dialectical materialism as a credible alternative to science. Through its category-theoretic reframing, I show that the postulate actually encapsulates two structurally different conditions, one of which directly constrains his own project while the other part captures a type of atomism juxtaposed with Badiou's condescending assessment of democracy.

By leveraging such claims indebted to his pre-committed, logicist comprehension of mathematics, Badiou reinstitutes an absolutist Truth-doctrine and an ultimately despotic political program against a genuine concept of multiplicity. His errors thus hardly constitute merely an inadvertent formal slip, but indicate a deeper, political intent.

\end{abstract}


\pagestyle{empty}
\begin{titlepage}
\pagestyle{empty}
\end{titlepage}
\maketitle	

\thispagestyle{empty}

\tableofcontents

\vspace{-.5cm}


\pagestyle{plain}
\pagenumbering{arabic}
\makeatletter
\def\@makefnmark{\hbox{\@textsuperscript{\normalfont\@thefnmark}}}
\makeatother

\section*{Introduction}
\setcounter{page}{1}

Alain Badiou, perhaps the most prominent French philosopher today, stands as one of the key figures in \textit{dialectical materialism}---a discourse that has deeply shaped the past century of continental philosophy through the vast influence of radical social critique\footnote{Badiou does not directly use the phrase \textit{dialectical materialism} to describe his own project, instead framing \LW\ around a `materialist dialectic' alongside a French formalist lineage running from Léon Brunschvicg through Jean Cavaillès, Albert Lautman, Jean-Toussaint Desanti, Louis Althusser, and Jacques Lacan.}. It offers a compelling point of contrast to contemporary scientific materialism.
Topos theory plays a pivotal role in this philosophical oeuvre---not least because it examines the basic, logic-grounded assumptions of mathematical sciences themselves. 

Yet, in Badiou's own endeavour, this role is purely implicit. Indeed, while arguing that mathematics `is ontology'\footnote{Badiou, Alain, \textit{Being and Event}, translated by Oliver Feltham, London and New York: Continuum, 2006. [Originally published in 1988.]}, Badiou discovers Paul J. Cohen's work on the incompleteness of set theory---a vital precursor and a prelude to the categorical, structural approach to the mathematical `topos'. Its configuration makes it evident how incompleteness, or what Badiou views as `errancy', is organized by geometric ideas---a fact examined more thoroughly in my second article on the topic\footnote{Veilahti, Antti (2026), `Bridging Badiou's Gaps: Topos Theory and the Bodies of Mathematical Expression', \textit{Preprint}, arXiv. \url{https://arxiv.org/pdf/2611.27280}}. However, Cohen's original presentation of the procedure, whose articulation was still indebted to classical set theory, remains the only discourse concerning the `subjectivity' of mathematics that Badiou genuinely respects. 

This dependence results in a profound ambivalence. On the one hand, even if implicitly, the most elementary structures of topos theory are already immersed in Paul J. Cohen's\footnote{Cohen, Paul J. (1963), `The Independence of the Continuum Hypothesis', \textit{Proc. Natl. Acad. Sci. USA}, vol. 50, no. 6. pp. 1143--1148; see also Cohen, Paul J. (1964), `The Independence of the Continuum Hypothesis II', \textit{Proc. Natl. Acad. Sci. USA}, vol. 51, no. 1, pp. 105--110.} forcing argument. Therefore, I can state that a `topos' secretly operates also in the first of Badiou's great treatises: \BE. William Lawvere and Myles Tierney introduced elementary topos theory in the 1970s while inspired precisely by those geometric structures or `topologies' inherent to forcing algorithms, ultimately introducing an entirely new way of embodying mathematical expression.

Nonetheless, Badiou's subsequent work, \LW\footnote{Badiou, Alain (2009), \textit{Logics of Worlds: Being and Event, 2}, translated by Alberto Toscano, London and New York: Continuum. [Originally published in 2006.]}, can be viewed as a continued attempt to defend the Hilbertian, set-theoretic approach as the sole legitimate basis for designating mathematical `being'. Consequently, the precise structural beauty of his first argument---Cohen's work operating as a dynamic prelude to subjectivity and change---becomes functionally condemned in his second treatise. Badiou, in effect, must have discovered topos theory because of this exact genealogical connection to Cohen's argument---an edifice to which \LW\ emerged as a localic, reactionary response. Symptomatically, the literature covered in his latter book remains confined to the so-called \textit{localic topos theory}---a limited area of topos theory that is still entirely reconcilable from the standpoint of classical set theory.

This reductionist orientation persists also in his later ontology, including \textit{The Immanence of Truths: Being and Event III}\footnote{Badiou, Alain (2022), \textit{The Immanence of Truths: Being and Event III}, translated by A. Toscano and K. Reinhard, Bloomsbury Academic.}, where his reoccurring insistence on transfinite set-theoretic analytics remains incapable of escaping the bivalued constraints of the `Platonic' approach to mathematical multiplicities. By anchoring the absolute within the hierarchical domain of large cardinals rather than shifting toward a relational geometry, his later endeavour thus also stops short of the genuine structural and semantic change that elementary topos theory carries. 

This nestling in localic theory is not as such a structural failure, of course: Cohen's argument can indeed be expressed as a localic topos\footnote{Veilahti, `Bridging Badiou's Gaps', 2026.}. However, it in no way undermines the precursors to change immersed already in Cohen's intuition, and ultimately in Badiou's own theory of subjectivity. Badiou's obsession with making a case against the relevance of categorical techniques thus appears evidently ironic. In particular, he asserts that the `same mathematics' still continues; that Grothendieck's geometric insight incorrigibly fails as a genuine event; and that there is no change in the `way of doing mathematics'\footnote{Badiou, \textit{Logics of Worlds}, 2009, p. 540.}---despite the fact that mathematicians themselves became intrigued by the very precursors that lend his philosophy its full force. 

This `ontological', formal restriction is constitutive from the very beginning of his second argument, and any conclusions he draws regarding materialism and phenomenology must be read with caution. The baseline of his thought comes down to a classical fallacy of \textit{affirming the consequent}: he assumes that demonstrating that certain categories take a particular localic form indicates that all of them must necessarily do so. 

This applies especially to what he labels as the `postulate of materialism'. Indeed, as I demonstrate in this intervention, should one follow topos theory proper, there exist two materially and mathematically distinct conditions that only locally---i.e., from the restrictive standpoint of \textit{localic} topos theory---amount to a unified postulate central to Badiou's phenomenological argument: that `every atom of appearance is real'. In other words, the hypothesis that a topos is `material' with respect to sheaf-theoretic formalism, but also that this materiality is bounded or local in an atomistic way.  Faithful to Badiou's own vocabulary, I call the former, stronger condition an \textit{atomistic} version of the postulate of materialism, whereas a general Grothendieck topos---a topos that can always be `materialized' as a sheaf-topos over the category of sets---satisfies only a \textit{weak} version of the postulate. The weak version is neither bounded nor logically constituted: this suggests that Grothendieck topoi, indeed, do make a `material' difference to set-theory.

Of course, associating Alexander Grothendieck's name even with this broader class of topoi is itself confusing, because Grothendieck mainly worked on much more delicate, `higher' abstractions of materiality that even the `weak' version of materialism cannot capture. Grothendieck-topoi form merely the very first example of constructs where geometric torsion becomes immanent in a way that renders logical or set-theoretic expressions of materialism obsolete. Yet precisely as such a first, `dialectical' step, these `weak' topoi are philosophically critical, particularly to dialectical materialism, because they already articulate a genuine transformation in the way of doing mathematics---a transformation that can be `materially' circumscribed from the point of view of set theory. In short, they challenge Badiou's ahistorical and `pure' description of mathematical, formal ontology.

To comprehend the structural distinction between these two doctrines of materialism, one requires a categorical point of view---an alternative foundation for mathematics, and especially an alternative perspective on the materiality of its \textit{expression}. While set theory is grounded upon logical `dialect', and can be considered dialectical in that precise respect, category theory is based upon \textit{structural arguments} and applies diagrams that possess a profoundly geometric meaning. 

In the context of post-structuralism---following Deleuze and Foucault---it is thus intriguing to treat the categorical approach as a `diagrammatic' one, operating parallel to the `dialectical' reason behind classical set theory. While topos theory subsequently bridges the two foundations by treating `dialectics' itself diagrammatically, its focus shifts \textit{from syntax to semantics}: ultimately to the question of the structural \textit{meaning} of contingencies inherent to what Badiou treats as the self-evident discourse of `being'\footnote{For further discussion, see the second article, Veilahti, `Bridging Badiou's Gaps', 2026.}. In particular, topos theory demonstrates that the problem of \textit{materiality itself is not inevitably a logically articulable question}.

As a brief historical overview, category theory was first introduced by Samuel Eilenberg and Saunders Mac Lane during the 1940s. As a result, various fields of mathematics, including algebraic geometry and geometric representation theory, became thoroughly reconceptualized. Philosophers, however, are only gradually paying attention to this conceptual shift\footnote{See Krömer, Ralf (2007), \textit{Tool and Object: A History and Philosophy of Category Theory}, Birkhäuser.} and its (post)\textit{structuralist} underpinnings\footnote{See for example: Awodey, S. (1996), `Structure in Mathematics and Logic: A Categorical Perspective', \textit{Philosophia Mathematica}, vol. 4, no. 3, pp. 209--237; Palmgren, E. (2009),  `Category Theory and Structuralism', url: \url{https://www2.math.uu.se/~palmgren/CTS-fulltext.pdf}; Shapiro, S. (1996), `Mathematical structuralism', \textit{Philosophia Mathematica}, vol. 4, no. 2, pp. 61-62; and Shapiro, S. (2005), `Categories, structures, and the Frege-Hilbert controversy: The status of meta-mathematics', \textit{Philosophia Mathematica}, vol. 13, no. 1, pp. 61--77.}.

Situated amidst the set-theoretic and categorical foundations of mathematical expression, topos theory stands at the crossroads of the two, `dialectical' and `diagrammatic' paradigms through which it intermediates the effects of change. To borrow Badiou's own, event-philosophical vocabulary, if category theory counts as an event, topos theory captures its `local consequences'---in fact, it embodies the two contrary dimensions of that localizing process that seeks to express change according to both `dialectical' and `diagrammatic' protocols. By combining these two paradigmatically different dimensions, it makes especially visible the \textit{process of localization} itself---something I indeed discuss in more depth in the second article\footnote{See Veilahti, Antti, `Bridging Badiou's Gaps', 2026.}.

Of course, what genuinely happened across this topos-theoretic `event' is not only category theory, but specifically a (self)-reflective interaction between two protocols of expression, and two alternative concepts of materiality. Either way, the two perspectives on materiality, and the inherent possibility of choice their recognition enables---i.e., whether the `dialectical' or `diagrammatic' expressions of materiality are conceived as primary---subjects dialectical materialism itself to interesting semiotic distinctions (such as the distinction between `designation' and `expression'). Even in Badiou's\footnote{Badiou, Alain, \textit{Logics of Worlds}, 2009, p. 334.} own vocabulary, it would be meaningless to ask whether one of the two perspectives, `top-down' and `bottom-up', came first. The crucial question instead is which one of these two framings of materiality, or two paradigmatic approaches to the event's `local configuration', points upwards. 

To introduce this `event', or rather a process of mathematical localization in the direction dialectical materialism itself labels `from below' (i.e., set-theory), Cohen showed that there are `undecidable' sentences independent of ZFC: there are contrary, transitive denumerable models of set theory. One model of set-theory---or one Cohen topos like I explore in the second article---verifies a certain statement, while another one disproves it. But lacking categorical insight---i.e., an alternative concept of materiality and locality---neither Cohen himself nor Badiou quite grasped how to articulate the existence of such contradicting models or situations concurrently. Without a properly relational formalism, neither could truly bridge the gap between the `top-down' and `from-below' perspectives---perspectives that in Badiou's\footnote{Badiou, \LW, 2009, p. 39.} work symptomatically translated to the `generic' and `local' characteristics of multiple-being, but which, in the context of topos theory proper, should rather be articulated as juxtaposing two alternative expressions of `from-below' (or locality) itself (each seeking to intervene the other `top-down').  

Badiou's philosophy could thus only deal with the singular occurrence, rather than concurrence, of inconsistency as a process. To him, it emerges as an effect of a `generic' decision, that is, as a choice of only one among all situations $S(\female)$. As a consequence, Badiou postulates it as the `fundamental law' that `one' needs to outright choose a single context or model to inhabit, instead of residing in and between many of them---a possibility that would entail an alternative approach to temporality that is not chrono-\textit{logical} but geometric in kind.

Topos theory, by contrast, uses categorical techniques to specifically express the amalgam of such situations so that the need to decide, in a point-wise fashion, does not arise---at least unless the situation is forced to materialize a binary `point'\footnote{See Veilahti, `Bridging Badiou's Gaps', 2026.}. This insight leads to a picture of Cohen's forcing argument that is dependent not on its outcome, but on its algorithmic structure itself as a subjectivating project. Such projections, like the second article explains, are in fact specific kinds of `geometric morphisms' that topos theorists, quite pertinently, themselves refer to as `points'. They constitute a formal correlate of what, under the rubric of dialectical materialism, it means to make a point (i.e., in a way that appeals to the logical `dialect' of classical topoi).

My question in this article specifically concerns how Badiou treats \textit{his own point}---i.e., the `postulate of materialism'---and how best to express and structure his own `project' in a categorically adequate way, irrespective of the chosen outcome of his `dialectical' decision. Ironically, Badiou\footnote{Badiou, \textit{Logics of Worlds}, 2009, p. 197.} himself seems to be `willing to bolster [his] confidence in the mathematics of objectivity', and even quotes Jean Dieudonné by saying that to earn the `right' to speak `one must master the active, modern mathematical corpus'\footnote{Badiou, Alain (2006), \textit{Briefings on Existence: A Short Treatise on Transitory Ontology}, translated by Norman Madarasz, State University of New York Press. 
p. 12.}, while he remains unable to comprehend his own `dialectical' project in a scientifically (i.e., category theoretically) sustained way.  

A deficit in this mathematical mastery, indeed, engenders several philosophical and also politically relevant ambiguities, which I analyze briefly in the concluding sections of this paper, as well as in the adjacent second article\footnote{See Veilahti, Antti, `Bridging Badiou's Gaps', 2026.}.
This intervention is organized as follows: Sections 1--4 interrogate Badiou's explicit formalism within his own parameters, defining his concept of `calculated phenomenology' and a world consisting of objects that mathematicians call sets graded over a specific semilattice $T$. Sections 5--7 shift to a properly categorical setting to re-formalize the schematic structures behind Badiou's work within a mathematically rigorous, yet genuinely relational framework. Sections 8--9 then revisit Badiou's original project, approaching it through sheaf theory while identifying the specific conditions captured by Badiou's `postulate of materialism'. This culminates in a purely categorical definition of an elementary topos and of the specific additional assumptions that would make it `material' in Badiou's sense of the term. Section 10, in turn, poses the question as to what extent Badiou's struggles simply reflect technical misunderstandings regarding the actual, `diagrammatic' mode of operation in category theory. Finally, after correcting Badiou's proof of the `subtle scholium' in Section 11, the last section elaborates the findings from a philosophical point of view. Section 12 particularly posits that Badiou's errors constitute not merely a technical omission, but that they serve the purpose of supporting his specific political agenda.

The mathematical definitions and categorical constructions deployed below constitute standard results within categorical logic, and for the most part can be found in Peter Johnstone\footnote{Johnstone, Peter T. (1977), \textit{Topos Theory}, Academic Press, 1977; see also Johnstone, Peter T. (2002), \textit{Sketches of an Elephant: A Topos Theory Compendium}, Volume 1, Clarendon Press.} and Saunders Mac Lane and Ieke Moerdijk\footnote{Mac Lane, Saunders and Ieke Moerdijk (1992), \textit{Sheaves in Geometry and Logic: A First Introduction to Topos Theory}, Springer-Verlag.}. Their mobilization throughout this text serves a dual purpose: expository and reconstructive. In particular, Badiou's philosophical matrix is translated into a contemporary categorical expression, while simultaneously rectifying the formal gaps and errors inherent in Badiou's own demonstrations.

\section{From Transitory Ontology to Categorical Phenomenology}

Since at least Kant's\footnote{Kant, Immanuel (1855), \textit{Critique of Pure Reason}, translated by J. M. D. Meiklejohn, Henry G. Bohn.} \textit{Critique of Pure Reason}, the question of causation has been pivotal to theoretical philosophy. What are the \textit{a priori} rules of deduction that can be assumed without the use of the senses? Arguing that these rules preclude any need for historical insight, Badiou\footnote{Badiou, \textit{Being and Event}, 2006, p. 4.} proposes that the question of the `pure' correlates best with formal logic and set theory: that `mathematics \textit{is} ontology---the science of being qua being'---even despite the dialectical `impasses of logic' manifested by Gödel, Tarski, and Cohen.

When it comes to ontology proper, Badiou thus claims, an event can only be localized against this language; there is no other way for the event to deploy its consequences but through those specific impasses whose communication remain indebted to the use of set-theoretic language. In this sense, even if the event concerns what mathematics is not, such a reductive understanding of mathematics becomes crucial to understanding precisely what that `what [it] is not' \textit{is}. Formally, Badiou then deploys the event as a self-belonging, auto-affirmative multiple $e \in e$. The crux of his argument is that the very existence of such an `event-multiple' obstructs the axiom of foundation\footnote{To demonstrate, if $e \in e$, then $y \in \{e\}$ implies $y = e$ so that one necessarily has $e \in y$ and $e \in \{e\}$, which implies $e \in y \cap \{e\}$. This directly contradicts the axiom of foundation.} and thus the very possibility of (mathematical) ontology.

\begin{axiom}[Foundation] 
For each non-empty set $x$ there is an element $y \in x$ so that their intersection $x \cap y = \emp$ is empty. 
\end{axiom}

However, the connection between this `ontological' event designated by the self-referential multiple and Cohen's argument remains `metaphorical' in Badiou's account. There is no mathematically rooted way to combine the two insights that occupy \BE---namely, the association between the event-multiple $e$ and the `generic' decision or set $\female$, which stands out as something `supernumerary' to a particular transitive, denumerable model of set theory $S$. 

Informed by this `gap', \LW\ commences from a different set of problematics, treating the question of `phenomenology' while deliberating on how precisely that germ of subjectivity appears `in the world'. One might wonder if this latter work successfully connects the two approaches to the `inconsistent', real being. Unfortunately, Badiou\footnote{Badiou, \textit{Logics of Worlds}, 2009, p. 39.} directly confesses that this is not the case. Comprehending what precisely \textit{happened} when Cohen deployed his method of `forcing' to introduce contradicting, non-reconcilable models of classical set-theory requires categorical techniques, which form another way to localize the consequences of the event, like the second stage of my argument conveys\footnote{Veilahti, `Bridging Badiou's Gaps', 2026.}. This structural framework coheres with the poststructuralist insight according to which what is crucial to an event is not the way it is unregulated in itself, self-referentially, but how it happens to itself: the event `makes the difference: not in space and time, but to space and time'\footnote{Fraser, Mariam, Kember, Sarah \& Lury, Celia (2005), `Inventive Life. Approaches to the New Vitalism'. \textit{Theory, Culture \& Society} 22(1): 1--14.}. Space and time are Kant's \textit{a priori} categories for the `trans-phenomenal real'. If an event is genuinely real, it must happen to them.

Within Badiou's `transitory ontology', driven by the absence of a genuinely categorical approach, this moment fades as he privileges set-theoretic language, like Norman Madarasz\footnote{Madarasz, Norman (2005), `On Alain Badiou's Treatment of Category Theory in View of a Transitory Ontology', In Gabriel Riera (ed), \textit{Alain Badiou---Philosophy and its Conditions.} New York: State University of New York Press. 23--44.} has pointed out. The only topoi Badiou subsequently deals with are \textit{localic topoi}: they are categories of $T$-sets, where the world's `transcendental indexing' $T$ is a set-theoretic poset. Within that restricted domain, phenomenology could only arise as a `calculated phenomenology'.


A lack of categorical insight, however, does not imply the absence of all categorical references. The revolution remains implicitly immersed within Badiou's own approach, even if in an unspoken manner. Indeed, by arguing that the formalism of \LW\ `is very different from the one found in \BE' as it shifts its emphasis from `\textit{onto}-logy' to `onto-\textit{logy}', Badiou explicitly refers to category theory\footnote{Badiou, \textit{Logics of Worlds}, 2009, p. 39.}---a structural framework that, at best, could mark a genuinely new `adventure of French philosophy'\footnote{Here, I refer to the title of Badiou, Alain (2012), \textit{The Adventure of French Philosophy}. Trans. Bruno Bosteels. New York: Verso.}. Category theory, indeed, deals with a fundamentally different formalism: an alternative structural approach to the `trans-phenomenal real'\footnote{Badiou, \LW, 2009, p. 104.}. Yet Badiou maintains that any difference it makes is specious, bearing no genuine difference, possessing no structural implications for space and time as concepts.

Badiou thus declines to execute the `Kantian' shift that would allow him to comprehend a mathematical object in a strictly categorical manner: not in terms of what it internally consists of, but through its sincere relations with other objects. He refuses to comply with what he himself regards as Kant's `sanctimonious declaration that we can have no knowledge of this or that'\footnote{Ibid., p. 535.} substantively, thus making his own `phenomenology' threaten him `with detention, the authorization to platonize'\footnote{Ibid., p. 536.}. Indeed, within this categorical mode of inscription, an object is always `empty'. Even the statement `$x \in A$' is not a matter of designation, but fuses a sheaf of relationships. The concept of element, in particular, must adopt an expression as an arrow between two objects, $X \to A$:  an element is not `owned' or possessed by the object ($A$). Instead, belonging marks a property ascribed by the entire category:

\begin{dfn}[Category]
To define a category $\ce$ formally, it consists of a collection of objects $\Ob(\ce)$ and a class (potentially a proper class within the ZFC axiomatic framework) of morphisms or arrows $\Hom(A, B)$ for any two objects $A, B \in \Ob(\ce)$, as well as identity arrows $\id_A: A \to A$ compatible with compositions. 
\end{dfn}

A category, rather than an isolated object, thus grounds the corporeal base for a new generation of mathematics. Much of Badiou's programmatic argument, by contrast, remains focused on that `Platonic' plane where objects themselves, rather than their categorical context, constitute incorporeal properties. This landscape subsequently allows relations to adopt at best a functional description. As a consequence, it is specifically within the modern `mathematical corpus' that Badiou's own framework does not fully integrate, as the categorical method shifts the focus away from functionality and toward \textit{functoriality}, which comprises secondary relationships between arrows (first order relationships) themselves. 

\begin{dfn}

A (covariant) functor between categories $F: \ce_1 \to \ce_2$ consists of a compatible collection of maps $$\Ob(\ce_1) \to \Ob(\ce_2) \quad \textrm{and} \quad \Hom(A, B) \to \Hom(F(A), F(B)).$$ Hence, not only does a functor transform objects into others, but it simultaneously translates their diagrammatic compositions across categories. 
\end{dfn}

Overlooking this structural transition, as Section 10 thoroughly illustrates, prevents Badiou's architecture from appreciating the philosophical value of functoriality, despite occasionally mentioning the `transcendental functor' which is a very specific type of functor: a sheaf. Namely, it is a functor from the world's `transcendental indexing' $T$ to the `Platonic' category of set-theory. Where Badiou's `calculated phenomenology' emerges solely around this single, specific functor, like Sections 2--4 shall illustrate, later sections of this study exemplify how category theory becomes invested in functoriality itself as a pivotal phenomenon. 

In categorical phenomenology, concepts like topology and even set-theory itself can thus only be introduced secondarily---not as \textit{a priori} entities, but as structures that individuate relationally from the class of objects, arrows, and functors between them. Even the concept of a point, like the second article portrays\footnote{See Definition 2.3 in Veilahti, `Bridging Badiou's Gaps', 2026.}, can be formalized as a pair of \textit{adjoint functors} between the ontological category $\sets$ and a relational topos $\Es$. 

Badiou, by contrast, treats topoi merely as if they were themselves static sets. Thus even the transformation from transitory ontology to calculated phenomenology remains a (philosophically) isolated point disconnected from a more broadly understood concept of change. The formative task of this article is to ask what might follow if `dialectical materialism' were to resist this categorical isolation and refused to be located only around a single, `Platonic' fixture. What would it mean, instead, to properly acknowledge and thus also exceed its `point' in a genuinely phenomenological sense? I shall now proceed to answer this question in a technically rigorous way. 

\section{Badiou's Ontological (Set-Theoretic) Reduction}

The technical architecture of the argument now commences by introducing the localic formalism that constitutes the explicit basis of Badiou's `calculated phenomenology'. Badiou is structurally unwilling to forfeit his core thesis that the history of the thinking of being (ontology) is identical to the history of mathematics, which he invariably equates with classical \textit{set theory}. It is consequently no accident that set-theory operates as the sole regulatory framework into which all topos-theoretic expressions are translated in his text. He refers to topoi merely in connection to the so-called complete Heyting algebras. Such an algebra is an organization expressed `internally' to any elementary topos, but which fails to exhaust its structure unless the topos itself is localic. In particular, Badiou's account leaves unarticulated the fact that there exist both `internal' and `external' interpretations of such lattice structures. 

More formally, an \textit{external} complete Heyting algebra $T$, as formalized below (Definition \ref{extHeytingDef}), generates the category of $T$-sets\footnote{In mathematical logic, these structures are traditionally termed $\Omega$-sets. However, I deliberately avoid a confusion between the internal categorical object $\Omega$ and the extensive grading of a locale resulting from a push-forward $T = \gamma_*(\Omega)$ in a bounded geometric morphism $\gamma: \ce \to \sets$.} that, unlike a general topos, conforms to the traditional parameters of (intuitionist) logic. Thus the spatial constitution of Badiou's worlds relies entirely upon his \textit{own programmatic decision} to cede to this narrow regime of objects---a choice that becomes pivotal to his strategic effort to denounce the relevance of category theory.

This problematic is explicitly visible in his designation of the world $\m$ as a `complete' situation of being---especially because this leads to the paradoxical attempt to capture this completeness in set-theoretic terms, thus spilling over to another concept: that of a `\textit{universe} [which is] the (empty) concept of a being of the Whole'\footnote{Badiou, \textit{Logics of Worlds}, 2009, pp. 102, 153--155.}. Badiou explicitly recognizes its `impostrous' nature; yet, in actual mathematical practice, his `world' is a sheaf-category over the `Whole' of set-theory and thus structurally \textit{exceeds} what, through its very designation, is supposed to capture the problem of the `Whole'. Thus it is not a surprise that his vocabulary becomes occasionally confusing: in certain registers, the algebra $T$ is termed the `transcendental of the world', as if $\m$ were defined strictly as a single specific locale, whereas elsewhere $\m$ refers more broadly to what looks like the category of all localic topoi ($\Loc$). Nonetheless, Badiou never makes the categorical leap beyond this (in set-theoretic sense) local understanding of (world's) completeness, that is, a configuration which is a conjunction-semilattice with extra structure ensuring its completeness (see Definition \ref{extHeytingDef}). 

\begin{dfn}
A \textit{conjunction-semilattice} is a set $T$ equipped with a partial order relation $\leq$, a minimal element $\mu \in T$, and a maximal element $M \in T$. It is further equipped with a conjunction operator $\wedge: T \times T \to T$ such that $p \wedge q \leq p, q$ is the greatest lower bound and $p \wedge q = q \wedge p$. These conditions also ensure the relational equivalence: $p \leq q$ if and only if $p \wedge q = p$. 
\end{dfn}

What Badiou further requires of world's `transcendental indexing' $T$ is that it has least upper bounds which make it \textit{complete} as a Heyting algebra.

\begin{dfn} \label{extHeytingDef}
A \textit{complete} external Heyting algebra $T$ is a conjunction-semilattice together with the least upper bound operator $\Sigma: \Power T \to T$ which is distributive with respect to $\wedge$. Formally this means that $$\Sigma A \wedge b = \Sigma \{a \wedge b \mid a \in A \}.$$
\end{dfn}

What makes the complete external Heyting algebra structure, and particularly the `envelope' $\Sigma$ paramount within intuitionist logic stems from the fact that propositional calculus can be constructed directly upon the operations formalized above\footnote{The standard notation for the least upper bound, or what Badiou refers to as the `envelope', is $\bigvee$ instead of $\Sigma$, and the implication-operator standardly uses the notation $\rightarrow$. However, to ensure continuity, this text employs the idiosyncratic notation used in \textit{Logics of Worlds}. 
}: 

\begin{dfn}[Deduction]
The material implication or dependence relation $\Rightarrow$ is an operator on $T$ satisfying the identity:
$$ p \Rightarrow q = \Sigma \{t \in T \mid p \wedge t \leq q\}. $$
\end{dfn}

\begin{dfn}[Negation]
The pseudo-complementation or negation $\neg: T \to T$ is a function formalized such that:
$$ \neg p = \Sigma \{q \in T \mid p \wedge q = \mu\}, $$
consequently satisfying the structural identity $p \wedge \neg p = \mu$.
\end{dfn}

\begin{rmk}
Please note that in terms of the subobject classifier $\Omega$---i.e., the diagrammatic correlate of Badiou's indexing---the minimal and maximal elements cannot be presented as static sets, but must be formalized as diagrams\footnote{In the specific case of \textit{localic Grothendieck topoi}---i.e., Grothendieck topoi that support a set of generators (see Postulate \ref{strong_postulate} )---the external Heyting algebra $T$ arises directly as a push-forward of the internal subobject classifier $\Omega$. }: they are elements generalized by the initial and terminal objects $0$ and $1$. The internal order relation $\leq_{\Omega}$, in turn, can be defined formally as the \textit{equalizer} of the conjunction $\wedge$ and the first projection map $\pi_1$: $$\xymatrix{\leq_{\Omega} \ar[r]^{e} & \Omega \times \Omega \ar@<0.5ex>[r]^{\wedge} \ar@<-0.5ex>[r]_{\pi_1} & \Omega.}$$
To exemplify the `diagrammatic' way of reasoning, its symmetry-property can be expressed by asserting that the following square commutes and forms a so-called pullback:
$$\xymatrix{
\leq_{\Omega} \cap \geq_{\Omega} \ar@{^(->}[r] \ar@{_(->}[d] \ar[dr]^{\Delta \circ \iota} & \leq_{\Omega} \ar@{^(->}[d]^e \\ 
\geq_{\Omega} \ar@{^(->}[r] & \Omega \times \Omega,
}$$
where the diagonal map is given by $\Delta \circ \iota$. 

Likewise, the envelope $\Sigma$ can be defined as the map $\Omega^t: \Omega^\Omega \to \Omega^1 \cong \Omega$, which is internally left adjoint to the map $\downarrow seg: \Omega \to \Omega^\Omega$ that takes $p \in \Omega$ to the characteristic map of $\downarrow(p) = \{q \in \Omega \mid q \leq p\}$\footnote{See Johnstone, \textit{Topos Theory}, 1977, pp. 147--148. }. 
\end{rmk}

\begin{rmk} Unlike within what Badiou\footnote{Badiou, \textit{Logics of Worlds}, 2009, pp. 183--188.} terms a `classical world'---i.e., what is standardly called a Boolean topos and where the law of the excluded middle holds---the negation operator $\neg$ is not required to be an involution or reversible in a general setting (i.e.,  $\neg \neg$ might not equal $1_{\Omega}$). However, within Badiou's restricted domain of localic topoi (i.e., topoi conditioned by both postulates of materialism, see Section \ref{two_postulates}), this classic involution holds if the \textit{internal axiom of choice}\footnote{See Johnstone, \textit{Topos Theory}, 1977, p. 141.} is also valid, that is, when all epimorphisms split, as is standard in the category of sets. However, the intuitionist relation $p \leq \neg \neg p$ remains universally valid across all topoi.
\end{rmk}

\section{Atomistic Objects}

Badiou criticizes the forms of `intuition' traditionally associated with multiplicities such as space and time. However, his own intuitions remain heavily constrained by classical set theory---which in light of category theory seems `transitory' as a \textit{discourse}, not just `ontology'. Following this constrained regime of expression, let me now analyze how Badiou himself encounters the question of `atoms' and materiality. 

If topos theory designates the internal `transcendental' object $\Omega$ relationally, the external, set-theoretic $T$-form reduces the classificatory question back into an incorporeal framework. It particularly admits a set-theoretical, explicit order-structure $(T, \leq)$ in contrast to the more abstract, relational definition as a diagrammatic equivalizer. What he calls `atoms' then emerge directly in terms of this `incorporeal', elementwise defined operator $\leq$: the `transcendental grading' that provides the `unity through which all the manifold given in an intuition is united in a concept of the object'\footnote{Badiou, \LW, 2009, p. 231.}.

Formally, within the parameters of an external Heyting algebra, this structures an entity $(A, \Id)$ where $A$ is a set and $\Id: A \times A \to T$ is a function satisfying specific foundational conditions.

\begin{dfn}[Identity and Equalization]
The identity relation denoted as $\Id: A \times A \to T$ satisfies two fundamental conditions for all elements within $A$:
\begin{enumerate}\label{idconditions}
\item symmetry: $\Id(x,y) = \Id(y, x)$
\item transitivity: $\Id(x, y) \wedge \Id(y, z) \leq \Id(x, z)$.
\end{enumerate} 
These axioms guarantee that the resulting `quasi-object' is strictly objective, in the sense of being distinguished within the subjective discourse born by $T$: `the differences in degree of appearance are not prescribed by the exteriority of the gaze'\footnote{Ibid., p. 205.}. In sheaf-theoretic terms, the $\Id$-map evaluates the domain in which two sections of a sheaf agree. 
\end{dfn}

\begin{rmk}
This analogous `identity'-function correlates directly to the structural \textit{equalization} procedure native to category theory. Identities can be structurally understood as generalized equivalence relations. Given two arrows $X \rightrightarrows Y$, an \textit{equalizer} (which exists in any topos, given the existence of finite limits, see Definition \ref{toposdefinition}) is an object $Z \to X$ such that both induced compositions $Z \to Y$ are identical. Given a topos-theoretic object $X$ and a domain of stage $U$, pairs of elements of $X$ over $U$ (i.e., arrows of the form $U \to X$), can be systematically compared and equivalized by the evaluation arrow:
$$\xymatrix{X^U \times X^U \ar[r]^-{eq} & \Omega^U},$$
which structurally internalizes identity analogous to Badiou's explicit map $\Id$\footnote{See Johnstone, \textit{Topos Theory}, 1977, pp. 39--40.}. 
\end{rmk}

Now it is possible to formalize the structural conditions of what Badiou labels an `atom of appearing'.

\begin{dfn}
An \textit{atom} is a function $a: A \to T$ defined on a $T$-set $(A, \Id)$ such that it satisfies the two following axioms for all $x, y \in A$:
\begin{enumerate}[leftmargin=5cm]
\item[(A1)] $a(x) \wedge \Id(x, y) \leq a(y)$
\item[(A2)] $a(x) \wedge a(y) \leq \Id(x, y)$.
\end{enumerate}
\end{dfn}

As expressed in Badiou's own vocabulary, an atom can be defined as an `object-component which, intuitively, has at most one element in the following sense: if there is an element of $A$ about which it can be said that it belongs absolutely to the component, then there is only one. This means that every other element that belongs to the component absolutely is identical, within appearing, to the first'\footnote{Badiou, \LW, 2009, p. 248.}.

These two properties in the definition of an atom are highly motivated by the theory of $T$-sets (or $\Omega$-sets in the standard terminology of categorical logic). A map $A \to T$ satisfying the first inequality is traditionally conceptualized as a `subobject' of $A$---or in Badiou's restricted domain a $T$-subset of $A$. The idea is that, given a $T$-subset $B \subset A$, one can construct the function:
$$ \Id_B(x) := a(x) = \Sigma\{\Id(x, y) \mid y \in B\}, $$
and it is straightforward to verify that the first condition (A1) is satisfied. In the opposite direction, for any map $a$ satisfying the first condition, the subset:
$$ B = \{x \in A \mid a(x) = \Ee x := \Id(x, x)\} $$
is clearly a valid $T$-subset of $A$.

To illustrate why these maps are called `atoms' specifically, the second condition (A2) states that the subobject $a: A \to T$ is structurally a \textit{singleton}. This concept stems from the topos-theoretic internalization of the singleton morphism $\{\cdot\}: X \to \Omega^X$, which determines a specific class of $T$-subsets of $A$ that correspond to the atomistic $T$-subsets. For example, in the case of an ordinary set $S$ and an element $s \in S$, the singleton $\{s\} \subset S$ represents a particular, `atomistic' type of subset of $S$.

The question of `elements' incorporated by an object can thus be expressed externally in Badiou's localic theory, but must be formalized `internally' in a general elementary topos. For the same reason, there exist two distinct ways for an element to be `atomistic': in the first sense, an `element depends solely on the pure (mathematical) thinking of the multiple', whereas the second sense relates it `to its transcendental indexing'\footnote{Ibid., p. 221.}. In general (elementary) topos theory, this distinction is slightly more complex\footnote{In purely categorical terms, an atom is an arrow $X \to \Omega$ mapping an element $U \to X$ into $U \to X \to \Omega$, and thus an element of the exponential object $\Omega^X$. But arrows $X \to \Omega$, given that $\Omega$ is the subobject classifier, correspond bijectively to subobjects of $X$. This localic characterization of a `subobject' is thus topos-theoretically justified. The `element' corresponding with a singleton is exactly the monic arrow from the subobject to the object. Only in Badiou's localic case does the order relation $\leq$ extend to a partial order of the elements of $A$ themselves, and a closed subobject turns out to be exactly the one generated by its least upper bound---its envelope---as Badiou's proof of the existence of the `transcendental functor' (i.e., a sheaf; see the next section) demonstrates. In other words, an atom in Badiou's sense is an atomistic subobject within the topos of $T$-sets. In the diagrammatic terms of general topos theory, it agrees with the so-called singleton map $\{\}: X \to \Omega^X$, which is the exponential transpose of the characteristic morphism $X \times X \to \Omega$ of the diagonal $\Delta: X \hookrightarrow X \times X$. Expressed in terms of the internal Mitchell--Bénabou language, one may translate the equality $x =_A y = \Id_A(x, y)$ into the modal statement $x \in_A \{y\}$, which gives structural insight into why such an `atom' should be regarded as a singleton. Similarly, if $a$ only satisfies the first axiom, the statement $x \in A_a$ would be semantically interpreted to retain the truth-value $a(x)$. The concept of an atom also relates to another critical aspect: atomistic subobjects correspond to arrows $X \to \Omega$, which allows one to interpret $\Omega^X$ as the power-object. Functorially, $\Power: \Es^{\op} \to \Es$ maps the object $X$ to $\Power X = \Omega^X$, and an arrow $f: X \to Y$ maps to the exponential transpose $\Power f: \Omega^Y \to \Omega^X$ of the composite: $$\xymatrix{\Omega^Y \times X \ar[r]^-{1 \times f} & \Omega^Y \times Y \ar[r]^-{ev} & \Omega,}$$ where $ev$ is the counit of the exponential adjunction. The power-functor is occasionally axiomatized independently, but its existence follows formally from the axioms of finite limits and the subobject classifier (see Definition \ref{toposdefinition}), reflecting the fact that these two aspects are separate categorical suppositions regardless of their combined positing within the traditional set-theoretic framework.}.

Badiou still requires a further definition in order to formally state his `postulate of materialism'.

\begin{dfn} 
An atom $a: A \to T$ is \textit{real} if there exists an element $x \in A$ such that $a(y) = \Id(x, y)$ for all $y \in A$. 
\end{dfn}

This definition gives rise to the foundational postulate inherent to Badiou's structural understanding of `democratic materialism'. 

\begin{pos}[Postulate of Materialism]\label{original_postulate}
In a $T$-set $(A, \Id)$, every atom of appearance is real.
\end{pos}

What this postulate designates is that there inherently needs to \textit{exist} an ontological element $s \in A$ for every suitable subobject that structurally (read categorically) \textit{appears} to satisfy the exact same conditions as the  singleton $\{s\}$. In other words, what `appears' materially, according to the postulate, is forced to `be' within the assumed, incorporeal situation of being. 

\begin{rmk}[Support generators] \label{sg_axiom}
Topos-theoretically, this formulation relates directly to the so-called \textit{axiom of support generators} (SG), which entails that the subobjects of the terminal object 1 within the topos
generate the entire category\footnote{See Johnstone, \textit{Topos Theory}, 1977; and Mac Lane, Saunders \& Ieke Moerdijk (1992), \textit{Sheaves in Geometry and Logic. A First Introduction to Topos Theory}. New York: Springer-Verlag.
}. In other words its global elements---morphisms of the form $1 \to X$---are structurally sufficient to determine and individuate any particular object $X$. 

Consequently, it is this highly specific constraint (SG) that is tacitly assumed by Badiou's rigid notion of the `unity' or `constitution' of `objects'. They are structurally united only by their atomistic contents---i.e., by what they internally `consist of'---rather than being composed via categorical relations. This ontological demand specifically forces Badiou to cross the transcendental boundary---the very line that Kant drew when he interrogated \textit{Quid juris?} or `Haven't you crossed the limit?', as Badiou\footnote{Badiou, \LW, 2009, p. 104.} translates.
\end{rmk}

But even without assuming the postulate itself---that is, when considering a weaker category of quasi-$T$-sets not required to make every atom `real'---the category of quasi-$T$-sets possesses a functor that maps any quasi-$T$-set $A$ into the corresponding quasi-$T$-set of singletons $SA$ via the assignment $x \mapsto \{x\}$. Here, $SA \subset \Power A$, where $\Power A$ represents the quasi-$T$-set of all quasi-$T$-subobjects, defined as all maps $A \to T$ that satisfy the first of the two structural conditions (A1) of an atom designated by Badiou. 

It can then be demonstrated that $SA$ itself, in fact, constitutes a \textit{sheaf} within which all atoms are inherently `real', thereby completing the quasi-$T$-set into a proper $T$-set that satisfies the `postulate of materialism'\footnote{In fact, the category of $T$-$\sets$ is fundamentally equivalent to the category of $T$-sheaves $\Sh(T, J)$ as I demonstrate in the next section. Also see Wyler, Oswald (1991), \textit{Lecture Notes on Topoi and Quasi-Topoi}. Singapore, New Jersey, London, Hong Kong: World Scientific. 
 p. 263.}. Within the formal language of $T$-sets, the `postulate of materialism' thus reduces to designating a structural equality between the domain $A$ and its completed sheaf-theoretic set of singletons $SA$. I demonstrate this explicitly in the next section.

The specific class of objects Badiou discusses can now be properly defined as those quasi-$T$-sets whose atoms are universally real; they give rise to what Badiou phrases as the `ontological category par excellence'\footnote{Badiou, \LW, 2009, p. 221.}.

\begin{dfn}
An object within the category of $T$-$\sets$ is a pair $(A, \Id)$ satisfying the above structural conditions such that every internal atom $a: A \to T$ is real. 
\end{dfn}

Next, despite the obvious formal omission of his original programmatic text, Badiou attempts to demonstrate that such `objects' indeed give rise to a valid mathematical category of $T$-$\sets$.

\section{Badiou's `Subtle Scholium': $T$-sets are Sheaves}\label{tsetisasheaf}

By following established foundational accounts\footnote{See Bell, J. L. (1988), \textit{Toposes and Local Set Theories: An Introduction.} Oxford: Oxford University Press; Borceux, Francis (1994), \textit{Handbook of Categorical Algebra. Basic Theory. Vol. I.} Cambridge: Cambridge University Press; Goldblatt, Robert (1984), \textit{The Categorical Analysis of Logic}. Mineola: Dover; and Wyler, \textit{Lecture Notes on Topoi and Quasi-Topoi}, 1991.}, Badiou attempts to demonstrate that $T$-sets defined over an external complete Heyting algebra are \textit{sheaves} (Theorem \ref{scholium}) and thus generate a so-called \textit{Grothendieck topos}---specifically, a topos of sheaves of sets over a locale $T$ (Theorem \ref{finalscholium}). These constitute what Badiou calls a `subtle scholium'---a classical result going back to Higgs\footnote{Higgs, Denis (1973), `A category approach to Boolean-valued set theory', unpublished manuscript, University of Waterloo.} and developed in full by Fourman and Scott\footnote{Fourman, Michael P. \& Scott, Dana S. (1979), `Sheaves and logic', in M. P. Fourman, C. J. Mulvey \& D. S. Scott (eds.), \textit{Applications of Sheaves: Proceedings of the Research Symposium on Applications of Sheaf Theory to Logic, Algebra and Analysis, Durham, 1977}, Lecture Notes in Mathematics 753, Berlin: Springer, pp.~302--401.}; what follows is  a reading of that theory in Badiou’s vocabulary of atoms and transcendental indexing.

 While $T$ is defined as a set, it can be  transformed into a small category by defining its elements as objects and the partial order relations between its elements as the underlying arrows. Within this framework, Badiou introduces one additional concept---that of `existence'---which mathematically captures the domain in which an element, conceived as a section of a sheaf, is defined: 

\begin{dfn}
The self-identity or existence parameter of an element within a $T$-set $(A, \Id)$ is defined as
$$ \Ee x = \Id(x, x). $$
\end{dfn}

\begin{prop}
From the symmetry and transitivity axioms of the identity relation $\Id$, it follows directly that:\footnote{See Badiou, \LW, 2009, p. 247.}
$$ \Id(x, y) \leq \Ee x \wedge \Ee y. $$
\end{prop}

\begin{thm}\label{scholium}
An object $(A, \Id)$, whose internal atoms are universally real, is a sheaf.
\end{thm}

The proper demonstration that every such an object $(A, \Id)$ constitutes a sheaf  requires the precise definition of three structural operators: localization, compatibility, and order.

\begin{dfn}[Localization] 
For any given atom $a: A \to T$, a \textit{localization} $a \res p$\footnote{While contemporary sheaf theory and topos theory traditionally deploy the restriction bar notation, such as $a\res {p}$ or $a\vert_U$, I deliberately retain Badiou's specific $\res$ symbol throughout this intervention. My aim is to maintain  compatibility and terminological continuity with the idiosyncratic notation of \textit{Logics of Worlds}.} onto an element $p \in T$ is the atom that satisfies the following condition for each element $y \in A$:
$$ (a \res p)(y) = a(y) \wedge p. $$ 
The fact that the resulting map $a \res p: A \to T$ satisfies the axioms of an atom follows from the fact that the conjunction operator $\wedge$ commutes with the order relation $\leq$. By virtue of the `postulate of materialism', this localized atom is itself represented by a unique element $x_p \in A$. Consequently, the structural localization of elements $x \res p$ within $A$ remains well-defined: $$\Id(x \res p, y) = \Id(x, y) \wedge p.$$
\end{dfn}

Two ordinary functions $f$ and $g$ are defined to be compatible if they coincide for every element within their intersection or common domain thus satisfying $f(x) = g(x)$. However, when one defines arrows solely in the categorical framework, a non-elementary definition is entailed, referring to the restriction-operator (i.e., $f\vert_{X \wedge Y} = g\vert_{X \wedge Y}$) analogous to Badiou's localization-operator $\res$. 

\begin{dfn}[Compatibility]
Within Badiou's formal apparatus, two internal atoms are defined to be \textit{compatible} if and only if they satisfy the identity:
$$ a \dagger b \quad \iff \quad a \res \Ee b = b \res {\Ee a}. $$ 
\end{dfn}

\begin{prop}
I already noted that the identity relation satisfies $\Id(a, b) \leq \Ee a \wedge \Ee b$. It is an immediate consequence of the compatibility condition that if $a \dagger b$, then the converse relation $\Ee a \wedge \Ee b \leq \Id(a,b)$ also holds, establishing a strict equality between the two metrics. This resulting identity can subsequently be deployed as an alternative structural definition of compatibility.
\end{prop}

\begin{proof}[Sketch] 
The converse implication entailed by this alternative definition requires a more involved proof. It must first be demonstrated that $$a \res {\Id(a, b)} = b \res {\Id(a, b)},$$ and that the localization operation is transitive\footnote{Transitivity means that $(a \res {p})\res {q} = a \res (p \wedge q)$; see Badiou, \LW, 2009, pp. 271--272.}), ensuring compatibility with the underlying lattice structure. 

Finally, once it is demonstrated that $a \res {\Ee a \wedge \Ee b} = a \res {\Ee b}$, the compatibility relation $a \dagger b$ follows as an immediate consequence\footnote{Ibid., p. 273.}. 
\end{proof}

\begin{rmk}
While sheaf theory generally requires a compatibility operator, the indexing category is not required to admit a global partial order structure $\leq$ like $T$ does. Instead, the assumptions of the order relation emerge effectively at the level of expression, via a relational intervention by a sieve. 
\end{rmk} 

\begin{dfn}[Order Relation] \label{order_def}
Formally, one can extend the partial order structure of $T$ to any object ($T$-graded set) $A$ so that it concerns sections of sheaves, rather than their underlying degrees of existence:  $a \leq b$ if and only if $\Ee a = \Id(a, b)$. It is straightforward to demonstrate that the relation $a \leq b$ is logically equivalent to the simultaneous conditions that $a \dagger b$ and $\Ee a \leq \Ee b$. Furthermore, the relation $\leq$ is reflexive, transitive, and anti-symmetric\footnote{Ibid., p. 258.}. 
\end{dfn}

I have now introduced sufficient structure to prove the earlier theorem stating that the objects in Badiou's category are actually sheaves. 

\begin{proof}[Proof of Theorem \ref{scholium}]
The explicit verification of the sheaf condition is constructed upon the equivalence of the following three conditions: 
\begin{eqnarray*}
a & = & b \res {\Ee a} \\  
a  \dagger  b & \textrm{and} &  \Ee a  \leq  \Ee b \\ 
\Ee a & =&  \Id(a, b).
\end{eqnarray*} 
The crucial fact to show is the equivalence between the first and third conditions, which follows from the identity\footnote{Ibid., pp. 273--274.} $$\Ee (a \res {p}) = \Ee a \wedge p.$$

To proceed with the demonstration, I connect these operations directly to the envelope $\Sigma$. First, it must be shown that if $b \dagger b'$, then the inequality $b(x) \wedge b'(y) \leq \Id(x, y)$ holds universally for all elements $x, y \in A$. This follows directly from the preceding framings. The critical phase is subsequently to demonstrate that the mapping:
$$\pi(x) = \Sigma \{ \Id(b, x) \mid b \in B\}$$ 
satisfies the conditions of an atom, provided that the elements constituting the region $B$ are pairwise compatible\footnote{Ibid., p. 263.}. This is because the region then retains a `real' element that materializes the atom. The first axiom (A1):
$$\Id(x, y) \wedge \pi(x) \leq \pi(y)$$ 
follows through standard lattice conditions\footnote{Ibid., p. 264.}.
Now, the identity resolves to $\pi(x) \wedge \pi(y) = \Sigma\{\Id(b, x) \wedge \Id(b', y) \}$, and by virtue of the preceding equations, $\Id(b, x) \wedge \Id(b', y) \leq \Id(x, y)$. Consequently, the metric $\Id(x,y)$ serves as an upper bound; since the envelope $\Sigma$ represents the least upper bound, the inequality $$\pi(x) \wedge \pi(y) = \Sigma \{\Id(b, x) \wedge \Id(b', y)\} \leq \Id(x,y)$$ follows. This suffices for verifying that $\pi$ constitutes an atom, and I denote its corresponding real element by $\epsilon$. 
It is now straightforward to demonstrate that $\Ee \epsilon = \Sigma \{\Ee b \mid b \in B\}$. The  element $\epsilon$ operates as the least upper bound of the region $B$---its real synthesis\footnote{Ibid., pp. 265--266.}.

Badiou characterizes this `transcendental functor of the object'\footnote{Ibid., p. 278.}---his name for a sheaf---as `not exactly a function', given that it associates the domain with subsets rather than singular elements. Indeed, a section of a sheaf, unlike a function, does not, by definition, map element to element, but rather exists as a relational node in an amalgam of compatible sections existing for each index ($U$ or $p$), whether they can be conceived as classical partial functions or not. Thus it generalizes the idea of a function, as a `synthesis' glued together from partial ones, in a diagrammatic way. In particular, it does not operate on sets but rather constitutes a set-theoretic representation of elements that `exist' over different indices $p$, each associated with a \textit{set} $\F(p)$, where $p$ can be an element in $T$ or an object in a more general index category $\ce$. The sheaf as a whole then constitutes a functor $\ce^{\op} \to \sets$, or alternatively an object in the power category $\sets^{\ce^{\op}}$ (i.e., a category of presheaves).

In Badiou's restricted framework, the index category $\ce_T$\footnote{The objects constitute the elements $\Ob(\ce_T) = T$. For any elements $p, q \in T$, the hom-sets are defined as $\Hom_{\ce_T}(p, q) = \{\leq\}$ if $p \leq q$, and $\Hom_{\ce_T}(p, q) = \emp$ otherwise.} is derived directly from the conjunction-semilattice $T$. To illustrate the sheaf-condition in this constrained domain, the functor ${\ce_T}^{\op} \to \sets$ comprises of the following. For a given object $A$, the sections of the sheaf form the set $\F_A(p) = \{x \in A \mid \Ee x = p\}$. If there exists an element $y \in \F_A(p)$, which means that $\Ee y = p$, then the structural equation $\Ee(y \res {q}) = \Ee y \wedge q$ resolves to $\Ee(y \res {q}) = p \wedge q$. If $q \leq p$, then $\Ee y \res {q} = q$, generating the following commutative diagram:
$$\xymatrix{
p \ar[r]^{\F_A} \ar[d]_{\geq} & \F_A(p) \ar[d]^{\cdot \res {q}} \\
q \ar[r]^{\F_A} & \F_A(q)
} $$ where the arrow $p \xrightarrow{\ \geq \ } q$ now goes in the opposite direction because it lies in the category $\ce_T^{op}$. This guarantees that $\F_A$ is a functor and thus structures a presheaf. 

Finally, one must verify the gluing condition that distinguishes a sheaf from a mere presheaf: this precisely is the contents of an atom encompassing a `real synthesis' in Badiou's vocabulary. To this end, the assignment $J(p) = \{\Theta \mid \Sigma \Theta = p\}$ forms a basis for a Grothendieck topology on the locale $T$---a fact discussed in more detail in Section \ref{sheaf_section}. Let me now consider a  compatible family of local sections $x_q$ where $q \in \Theta \in J(p)$. If these elements derive from an underlying global section $x_p$, they must be pairwise compatible, thus satisfying the standard matching condition $x_q \res (q \wedge q') = x_{q'} \res (q \wedge q')$, or alternatively denoted as the compatibility relation $x_q \dagger x_{q'}$. 

The structural objective is thus to determine a unique element $x_p \in \F_A(p)$ such that $x_q = x_p \res {q}$ holds for all $q \in \Theta$. But one may choose $x_p$ to be the envelope $\Sigma \{x_q \mid q \in \Theta\}$. I have already demonstrated that the envelope $\Sigma\{x_q\}$ localizes precisely to $x_q$ for all $q$, satisfying the identity:
$$ \Sigma\{x_q\} \res {q} = x_q, \qquad \forall q \in \Theta, $$
while ensuring that $\Ee \Sigma \{x_q\} = p$. Thus it satisfies the required gluing condition. It is also straightforward to verify that it is unique, completing the proof that these objects, which Badiou calls sets localized `in the space of [their] appearing'\footnote{Ibid., pp. 225--226.}, correspond to those constituting the topos of sheaves $\Sh(T, J)$.
\end{proof}

\begin{rmk}\label{proofgap}
What Badiou disregards concerning the definition of an object relates to whether the region $B \subset A$ itself, comprising pairwise compatible elements, is an object in the category. He successfully demonstrates that the mapping $\pi(x) = \Sigma\{ \Id(b, x) \mid b \in B\}$ structures a proper atom of $A$. If one seeks to construct the smallest subobject closure $\bar B$ containing the region $B$ that itself constitutes a proper object, the real element $\epsilon$ representing the atom $\pi$ must, by virtue of the postulate of materialism, reside within the atom itself: hence $\epsilon \in \bar B$. 
Consequently, because all intermediate elements are compatible, every element $b' \in \bar B$ satisfies  $b' \leq \epsilon$. As a consequence, the (sub)objects of $A$ are generated by the principal down-sets $\downarrow(\epsilon)$. 
\end{rmk}

\section{The Categorical Shift and Badiou's Concept of Materiality} \label{sheaf_section}

As analyzed in the preceding sections, category theory---the relational means of reassigning the `unassignable'---reorients mathematics away from the classical questions of contents and consistency, instead appreciating problems of composition, coherence, and `being-there'. Especially the latter, folding around the question of `taking place', is a central concern in topos theory. \textit{Topos} is originally a Greek phrase explicitly referring to a `place' or a `commonplace', thus offering an academic expression of profound phenomenological concepts. 

This formulation contests Badiou's insulated framing of `calculated phenomenology', whose formal restrictions I delineated above. However, instead of outright replacing the question of being with that of `being-there' (and ultimately becoming), in the second article I demonstrate that topos theory forces these two discourses to  \textit{intersect}: if it is not a phenomenological ontology \textit{per se}, at least it provides a formal re-expression of interesting philosophical phenomena, such as Hegel's divided `Church' of logic or Heidegger's \textit{Da-Sein}, not to mention the indefinite limit structures that emanate Deleuze's virtual aspects of becoming. 

It thus remains critical to comprehend sheaf theory both to clarify Badiou's own programmatic aims and to introduce an alternative, ahierarchical (or at most locally hierarchical) definition of Grothendieck topoi, which subsequently inspired the broader, relational definition of elementary topoi as well as higher topoi emerging as their infinite limit. A Grothendieck topos, indeed, is a specific type of topos whose objects are sheaves. And a sheaf, like I already illustrated in the previous section, is a functor that composes local sections so as to diagrammatically represent the functional dynamics of a geometric space without reducing it to the atomistic level of set-theoretic, incorporeal elements.

\begin{dfn}[Sheaf over a Space] \label{sheaf_over_space_def}
If $X$ is a classical topological space, then for each open set $U \subset X$, a sheaf $\F$ consists of an object $\F(U) \in \Ob(\ce)$ satisfying two primary conditions. First, there must exist a family of canonical restriction morphisms:
$$ \F(U) \to \F(U'): f_U \mapsto f_U\vert_{U'} $$
for all open inclusions $U' \subset U$. This makes it a functor $\Ou(X)^{\op} \to \sets$, where $\Ou(X)$ is the frame of open subsets. Second, the definition requires a \textit{gluing condition}: given a family of local sections $f_i \in \F(U_i)$ that coincide on intersections, there exists a \textit{unique} global section $f \in \F(\bigcup_i U_i)$ such that $f\vert_{U_i} = f_i$ for all indices $i$. 
In categorical language, this gluing condition entails that the following diagram constitutes an equalizer for each covering sieve $(U_i)$ of the domain $U$:
$$ \F(U) \to \prod_i \F(U_i) \rightrightarrows \prod_{i, j} \F(U_i \cap U_j). $$
\end{dfn}

\begin{dfn}[Representable Sheaves] \label{naturaltransformation}

By virtue of the Yoneda lemma, the structural transition from classical indexing spaces $\Ou(X)$ to categories of presheaves on a general small category $\ce$ is entirely natural. Given any object $A \in \Ob(\ce)$, the hom-functor $h_A := \Hom(\cdot, A)$ is a presheaf by definition. A given sheaf $\F: \ce^{\op} \to \sets$ can be called `represented' by an object $A$ if there is a functorial isomorphism $\F \cong h_A$. This yields the Yoneda embedding:
$ \mathbf{y}: \ce \to \sets^{\ce^{\op}}, $ which maps any object $A \in \Ob(\ce)$ to the corresponding representable presheaf $h_A$\footnote{The sections of $\F$ over $A$, according to the Yoneda lemma, represent `natural transformations' between $h_A$ and $\F$---i.e., maps between functors themselves---so $\F(A) \cong \Nat(h_A, \F)$. 
For instance, a section $s \in \F(U)$, where $U \subset X$ is an open neighborhood, should correspond to a natural transformation $h_U \to \F$. Here, $h_U(U')$ consists of the single element $\{U' \subset U\}$ corresponding to the inclusion $U' \subset U$, as the index category $\ce$ is defined as the lattice $\Ou(X)$ with $\Hom_{\Ou(X)}(V, U)$ consisting of at most a single element (depending on whether $V \subset U$ holds).   Now, for any $U'' \subset U' \subset U$, this corresponding natural transformation yields a commutative diagram: 
$$\xymatrix{
h_{U'}(U) = \{U' \subset U\}  \ar[r]^{} \ar[d] & \{s\vert_{U'}\} \subset \F(U') \ar[d]  \\
h_{U''}(U) = \{U'' \subset U\} \ar[r]^{} & \{s\vert_{U''}\} \subset \F(U'').
}$$ }.
\end{dfn}

\begin{rmk}
Definition \ref{sheaf_over_space_def} is grounded upon the Bourbakian, set-theoretic definition of a topology as a collection of open subsets $\Ou(X)$, which naturally structures a locale. Topos theory, however, establishes an alternative categorical definition of topology (and of sheaves), which applies to any index category $\ce$---not just the semilattice $\Ou(X)$ or $T$. The second article explores this transition in more detail. Moreover, where classical topology begins with `points' and distills topological structures above them, topos theory begins with a relational approach to localization structures and seeks to distill ontological points, as structural embeddings of the `ontological' category $\sets$. Secondarily, as a synthetic result of functoriality itself, sheaf theory fills a mediating role, illustrating how these points `materialize' through the functorial configuration of the underlying relational structures: instead of assuming the designated objects as carrying some incorporeal, `Platonic' substance (or contents), topos theory effectuates points as an ontological residue of their functorial \textit{expression}, thus as an effect of description rather than as an immanent feature of those multiplicities themselves. 

Inasmuch as sheaf theory thus connects the categorical register with the ontological sphere of sets, this is true because of the definition of a $1$-category involves a stratum of $\Hom$-\textit{sets}. For a more general class of categories, where $\Hom$-objects themselves involve relational strata, like in the theory of groupoids and stacks (categories fibred on groupoids), this is not the case. In fact, Grothendieck extensively worked on these `higher categories' despite his name being so strongly anchored to ordinary sheaf topoi. Nonetheless, by defining a Grothendieck-topology based on the concept of sieves, sheaf theory itself admits a more categorical formulation. 
\end{rmk}

\begin{dfn}[Sieve] \label{sieve}
Let $\ce$ be a small category. A \textit{sieve}---or in French `crible'---on an object $C \in \ce$ is a collection of arrows  with codomain $C$ that is closed under composition from the left. However, to maintain connection to Badiou's order-theoretic formalism, if the index-category $\ce$ is a poset it can alternatively be defined as a covering family of $C$ that is downward closed. \end{dfn}

\begin{dfn}[Grothendieck Topology]
A \textit{Grothendieck topology} on $\ce$ is a function $J$ that assigns to each object $C \in \ce$ a collection $J(C)$ of sieves, called covering sieves, such that:
\begin{enumerate}
\item (identity) the maximal sieve $\{f \mid \cod(f) = C\} \in J(C)$,
\item (stability) if $\Si \in J(C)$, then the pullback $$h^*(\Si) := \{g \mid h \circ g \in \Si\} \in J(D)$$ is a sieve for any morphism $h: D \to C$, and
\item (transitivity) if $\Si \in J(C)$ is a covering sieve and $R$ is any sieve on $C$ such that $h^*(R) \in J(D)$ for all arrows $h: D \to C$ belonging to $\Si$, then $R \in J(C)$.
\end{enumerate}
The stability condition ensures that the inverse image of a covering sieve remains a cover. Given these  axioms, a topology can itself be formalized as a sheaf on the maximal category of sieves. 

\end{dfn}

\begin{rmk}
A Grothendieck topology enables the generalization of the sheaf concept to an expansive class of categories beyond the classical domain of Kuratowski topological spaces. If $\esp$ denotes the category of topological spaces and continuous maps, sheaves on a space $X$ are equivalent to the full subcategory of slices $\esp/X$ constituting local homeomorphisms $p: E \to X$. 

The corresponding pair of adjoint functors $\sets^{\ce^{\op}} \rightleftarrows \esp/X$ yields an associated sheaf functor $\sets^{\ce^{\op}} \to \Sh(X)$, demonstrating that sheaves can be conceptualized either as presheaves satisfying strict exactness limits or as spaces operating via local homeomorphisms\footnote{See Johnstone, \textit{Topos Theory}, 1977, pp. 10--11.}.
\end{rmk}

\begin{rmk}
In the case of a discrete Grothendieck topology $J_0$ (where the only covering sieve on an object $C$ is the maximal sieve), the category of sheaves $\Sh(\ce, J_0)$ is equivalent to the full category of presheaves $\sets^{\ce^{\op}}$. In a general non-trivial setting, however, the sheaf-category $\Sh(\ce, J)$ constitutes a reflective full subcategory of $\sets^{\ce^{\op}}$, where the embedding functor preserves the exact internal semantics of the topos even when the underlying presheaves fail to satisfy the two global gluing conditions.
\end{rmk}

Finally, operating on a covering basis, rather than a full topology, is often practically advantageous, like Badiou's elaborations implicitly demonstrate. 

\begin{dfn}
A coverage basis $K$ for a Grothendieck topology $J$ consists of collections of morphisms satisfying the following conditions:
\begin{enumerate}
\item for any isomorphism $f: C' \to C$, the singleton $\{f\} \in K(C)$,
\item if $\{f_i: C_i \to C \mid i \in I\} \in K(C)$, then the pullbacks belong to the family $\{\pi_2: C_i \times_C D \to D \mid i \in I\} \in K(D)$ when taken along any given morphism $g: D \to C$, and
\item if $\{f_i: C_i \to C \mid i \in I\} \in K(C)$, and if for each index $i \in I$ there exists a covering family $\{g_{ij}: D_{ij} \to C_i \mid j \in I_i\} \in K(C_i)$, then the composite family $\{f_i \circ g_{ij}: D_{ij} \to C \mid i \in I, j \in I_i\}$ belongs to $K(C)$.
\end{enumerate} 
For any such basis $K$, one defines the covering sieve as $J_K(C) = \{\Si \mid \Si \supset R \in K(C)\}$, which rigorously generates the underlying topology.
\end{dfn}

Now I have introduced the formal apparatus required to accomplish the final step of Badiou's proof analyzed in the preceding section (see Remark \ref{proofgap}). Namely, the category of presheaves $\sets^{T^{\op}}$ is generated by the representable presheaves of the form $\mathbf{y}(p) = h_p = \Hom_T(\cdot, p)$ by virtue of the Yoneda lemma. Moreover, let $$ a: \sets^{T^{\op}} \to \Sh(T, J) $$
be the associated sheaf functor mapping $P \mapsto P^{++}$\footnote{The formal construction of the associated sheaf functor proceeds by defining the plus-construction: $$ P^+ = \varinjlim_{R \in J(C)} \Match(R, P), $$ where $\Match(R,P)$ denotes the matching families over the sieve. When executed twice, the assignment yields the exact left adjoint morphism $a: \sets^{\ce^{\op}} \to \Sh(\ce, J)$, which maps a presheaf to its sheafification.}.
Because the lattice $T$ is an ordered poset, for any element $p \in T$, the representable arrow $h_p  \hookrightarrow 1$ is a monomorphism; since the associated sheaf functor $a$ is left-exact, the resulting morphism $a\mathbf{y}(p) \hookrightarrow 1$ is likewise a monomorphism within the topos\footnote{See Mac Lane and Moerdijk, \textit{Sheaves in Geometry and Logic}, 1992, p. 277.}. Through this factorization, every subobject sheaf can be evaluated as a subobject of the terminal object $1$. 

This is the exact contents of what is called the \textit{axiom of support generators} (see Remark \ref{sg_axiom}) that remains absolutely crucial to Badiou's constitutive `postulate of materialism'. It follows explicitly from the organization of $T$ as an ordered poset and from this order relation being functorially extendible to the general objects constituting the topos $\Sh(T, J)$, like I observed in Definition \ref{order_def} in terms of Badiou's own formalism.

\section{Historical Value of the Geometric Turn}

But what renders categories historically valuable, and, in particular, what demonstrates that the categorical transformation is conceptually genuine? Of course, as I have done, it is one thing to demonstrate how Badiou's programmatic formalism fails to grasp the full colors of categorical change. But it is another matter altogether to demonstrate that mathematics itself changes, and that the `Platonic', \textit{a priori} core of Badiou's endeavor remains structurally insufficient. For this, one requires a more tangible, historical insight---something leaning towards \textit{empiricism}.

Yet, like Deleuze's transcendental empiricism conveys, the empirical does not necessarily operate in direct opposition to mathematics. Rather, it relates directly to results that stem from, and would have been impossible to comprehend without category theory. It is through experience that one is taught the structural meaning and deployment of categories---an experience obviously absent from Badiou's own habituation within mathematics.

To contrast, Alexander Grothendieck opened up an entirely new regime of algebraic geometry by generalizing the classical notion of space---first scheme-theoretically via sheaves, and subsequently in terms of groupoids and higher categories\footnote{Higher categories remain significantly more `Grothendieckian' in their non-local interactive properties than the standalone Grothendieck topoi themselves.}. Topos theory became synonymous with the study of categories satisfying the foundational Giraud axioms\footnote{See Johnstone, \textit{Topos Theory}, 1977, p. xii.}, rooted within Grothendieck's geometric machinery. 

By utilizing these non-local tools, Pierre Deligne proved the celebrated Weil conjectures---the mod-$p$ analogues of the Riemann hypothesis. These conjectures---anticipated already by Gauss---concern the local $\zeta$-functions derived from evaluating the number of points of an algebraic variety over a finite field: an algebraic structure structurally analogous to the field of rational numbers $\Q$ or real numbers $\R$, but containing only a finite number of elements. 

Given that algebraic varieties are represented in polynomial terms, it is possible to analyze geometric structures analogous to the Riemann hypothesis over such finite fields $\Z/p\Z$ (i.e., the whole numbers modulo $p$). Such varieties, which seem `ontologically' discrete but arithmetically integrated, had previously been completely excluded from topological and geometric inquiry; yet they now appeared to embody structural continuities invisible to the gaze of differential geometry. In particular, thanks to Grothendieck's insight, geometry was no longer overshadowed by a defensive need to decide between `discrete' and `continuous' modalities of the subject---a dualistic boundary that Badiou's framework still treats as an absolute separation.

Along with continuous manifolds, such seemingly `discrete' varieties could subsequently be investigated based on constructs like Betti numbers. Analogous to the processes which Cohen's forcing argument makes manifest across apparently isolated models of set-theory, like my second article more thoroughly investigates, even finite varieties encompass topological \textit{precursors} unreached by the classical analytical formalism. Likewise, the development of \textit{étale cohomology} allowed topological concepts, such as open neighborhoods, to be deployed in the context of algebraic geometry, whose classical Zariski topology was too coarse to permit a meaningful geometric interpretation. 

Introducing these concepts on the basis of Jean-Pierre Serre's suggestions, Grothendieck revolutionized the field of geometry---a transformation directly enabling Deligne's proof of the Weil conjectures, not to mention Andrew Wiles' proof of Fermat's Last Theorem. This insight drew upon the observation that if morphisms of varieties were evaluated by their functorially adjoint field of functions, it was possible to consider geometric morphisms as dual to algebraic ones. The algebraic category itself was too restrictive, however, because field morphisms are always monomorphisms, which forces the corresponding geometric morphisms to operate as epimorphisms. To generalize the notion of a neighborhood to the algebraic category, algebraic fields were thus embedded into the larger category of commutative rings, resulting in scheme theory. 

To explain the topological value of this shift, while a traditional Kuratowski covering space is locally split, the same is not true for the dual category of fields. The category of fields specifically lacks an operator analogous to pullbacks (fiber products) unless it is evaluated as being embedded within commutative rings, for which pullbacks possess a dual operation expressed by the tensor operator $\otimes$. Grothendieck thus realized he could replace incorporeal, contained neighborhoods $U \hookrightarrow X$ with a relational description: as maps $U \to X$ that are not necessarily monic, but which correspond to ring morphisms instead.

Topos theory applies a similar insight, but expands it beyond specific varieties to encompass the entire  theory of sets. Ultimately, it was thanks to Cohen's work that William Lawvere and Myles Tierney realized its immense importance for classification and truth as logical concepts. In the classical category of sets, the classification of elements between two sets reduces to a binary question: does this element belong to a given set or not? True or false. But this is not the case within a general topos, where the composition of the subobject classifier is geometric and captures all the relations (even torsion) in between.

Indeed, through the notion of a characteristic arrow, Lawvere and Tierney framed belonging itself as a categorical \textit{relationship}, without reducing it to the `Platonic' notion of internal contents or `being'. It was the structural description of that characteristic arrow---which they pertinently labeled as the arrow `true'---that marked the historical beginning of \textit{geometric logic} and topos theory. They rephrased the binary complete Heyting algebra of classical truth with the categorical subobject classifier $\Omega$, formalized as an internal object satisfying a specific pullback condition. The crux of elementary topos theory was securing these axioms relationally, so as to fully formalize topos theory independently of set-theoretic axioms.

I will return to this definition in the subsequent sections. But to comprehend its significance as a vital link between geometry and formal language, it is useful to see how the characteristic map behaves within classical set theory. In particular, by expressing truth in this manner, it became possible to reduce the Axiom of Comprehension\footnote{\label{axiomofcomprehension}If $\lambda$ is a formal property, then the set discerned by this property exists. In set-theoretic formalism, it is expressed as: $\forall w_1, \ldots, w_n \forall A \exists B \forall x (x \in B \iff [x \in A \wedge \lambda(x, w_1, \ldots, w_n, A)])$, where $x, w_1, \ldots, w_n$ are free variables.}---which entails that any suitable formal condition $\lambda$ gives rise to a specific set $\{x \in A \mid \lambda(x)\}$---to a statement regarding adjoint functors\footnote{See Johnstone, \textit{Topos Theory}, 1977, p. xiii.}.

While elementary topos theory offered a new, geometric foundation of mathematical expression, for many practical applications a more specific class of Grothendieck topoi turned pivotal. There, too, the `way of doing mathematics' changes fundamentally, insofar as the subobject classifier is categorically defined and there is no absolute primacy of the empty set (initial object); instead, mathematics constructs itself from the terminal object $1$. However, in some of them---those that fulfil what I call the weak postulate of materialism---there appears a `material' way to express the structural difference such topoi make in terms of set theory: these topoi can always be expressed as categories of sheaves, $\Sh(\ce, J)$, where $\ce$ is a small category equipped with a specific Grothendieck topology.

\section{The Postulate of Materialism---Or Two Postulates In Fact}\label{two_postulates}

Let me now examine the postulate of materialism from a categorical point of view, demonstrating how and why geometry matters to the question of materiality---even when this question is anchored to classical set theory. Therefore, to understand what Badiou's own definition constrains, I will explicate the precise condition that enables us to specify whether what I call a `weakly material' topos is localic and atomistic, that is, constitutively material in Badiou's strong sense of the postulate. Therefore, I begin with a more general, abstract expression of Badiou's objects in terms of elementary topos theory, even if they further satisfy the more specific conditions of localic theory, thus reducing to a sheaf-category over an externally graded algebra $T$. 

Indeed, Badiou's `worlds' may always be interpreted as Grothendieck topoi---even if in general they do not materialize in $\sets^{T^{\op}}$ but rather in a more general sheaf category materialized over an index category $\ce$ that may itself exhibit \textit{torsion}, thus effectuating such \textit{internal endomorphisms} that violate Badiou's `\textit{fundamental law of the subject}'\footnote{Badiou, \textit{Being and Event}, 2006, p. 401.}, that is, a transcendental prohibition of any grain of torsion. To constitute such a `weakly' material sheaf category---i.e., a Grothendieck topos---an elementary topos $\Es$ must retain a `geometric morphism'\footnote{For a definition of this morphism, see Veilahti, \textit{Bridging Badiou's Gaps}, 2026.} $\Es  \xrightarrow{\;\gamma\;} \sets$ that materializes its \textit{internal} `transcendental' object $\Omega$ as a set $\gamma_*(\Omega)$, even if this morphism $\gamma$ is not bounded in the `atomistic' sense of the internal transcendental structure being `bounded' by this material representation mediated by $\gamma$. This implies that while the local sections `materialize' within the category of sets, the geometric structure of the topos as a whole remains irreducible to this material surface.

\begin{pos}[Weak Postulate of Materialism]\label{weakpostulate}
Categorically, the weaker version of the postulate of materialism stands precisely for the geometric condition that transforms an elementary topos (Definition \ref{toposdefinition}) into a Grothendieck topos. In particular, it assumes a topos $\Es$ to be equipped with a global `materialization' morphism $\Es \to \sets$, which enables it to be described as a category of sheaves. 
\end{pos}

Next I analyze how, and especially on what further conditions, this weak postulate of materialism relates to how Badiou himself conceives the `postulate of materialism'---i.e., compatible with the original formulation \ref{original_postulate}. Indeed, in Badiou's framework, the internal transcendental structure $\Omega$ is not only materializable in sets, but bounded by this materiality, ultimately conditioning it as an external poset structure like $T$ that prevents such non-trivial endomorphisms or loops at the `transcendental' layer of existence. 

What remains a crucial aspect of objects in Badiou's `worlds' is, indeed, that they are atomistic---the precise assumption that every internal atom must be real. To reflect its true intent, I rephrase this ontologically stronger condition as the \textit{postulate of atomism}. Within elementary topos theory, it is formalized as the axiom concerning the \textit{support of generators} (SG). Let me now overview its categorical definition: a condition that accounts for the structural unity and singularity of the terminal object $1$. 

In a general (elementary) topos, the global elements of an object $X$ may be defined as arrows of the  form $1 \to X$, where 1 is the so-called `terminal object' in the category (i.e., every other object $Y$ retains a unique arrow $Y \to 1$). But, unless the topos is what I call strongly material or atomistic, it simultaneously possesses local elements $U \to X$ that cannot be determined by global elements alone. For each such arrow $U \to X$---i.e., a local element of $X$ at stage $U$---the pullback object $U^*X$ within the local slice topos $\Es/U$ possesses exactly one global element, which corresponds to a local section of $X$ over $U$\footnote{See Johnstone, \textit{Topos Theory}, 1977, p. 39.}. Eventually, these local elements embody non-trivial information of the object $X$ irreducible to its global elements if there occurs internal torsion that obstructs the dominant, `atomistic' hierarchy captured as $T$. 

This explicit, set-theoretically `material' poset structure, indeed, remains both the crux and the limit of Badiou's `calculated phenomenology'; it is analogous to Cohen's forcing structure backing Badiou's   `fundamental law'\footnote{Badiou, \textit{Being and Event}, 2006, p. 401.} like my second article conveys. It is only if such a rigid hierarchy is given once and for all that geometric intuition becomes superfluous. But in the case of a Möbius strip, for example, the non-trivial $\Z/2\Z$-bundles generate a category that cannot be similarly forced into a hierarchical poset.

On the basis of whether all elements of an object $X$ are fully discerned by its global elements, one can now separate the second aspect of Badiou's postulate in categorical terms---the side of the postulate that specifically relates to its `atomism'. Of course, unless both conditions are combined, it does not yet yield the complete `postulate of materialism' like Badiou originally phrased it, which assumes both conditions. 

\begin{axiom}[Postulate of Atomism]
An elementary topos $\Es$\footnote{See Definition \ref{toposdefinition}, p. \pageref{toposdefinition}.} \textit{supports generators} (SG) if the subobjects of the terminal object $1$ within the topos generate the entire category. This ensures that,  given any parallel pair of distinct morphisms $f \neq g: X \rightrightarrows Y$, there exists a subobject $V \hookrightarrow 1$ and an arrow $x: V \to X$ such that $f \circ x$ and $g \circ x$ fail to coincide; that is, \textit{all relational differences} in an atomistic topos are \textit{globally visible}, which makes this definition philosophically compatible with Badiou's sense of `democratic materialism'. 
\end{axiom}

\begin{rmk}

If a topos $\Es$ satisfies (SG)---i.e., the postulate of atomism---then the subobject classifier $\Omega$ operates as a cogenerator of $\Es$. This justifies designating it as an `internal atom' of the topos: whereas the subterminal generators serve as internal test-elements distinguishing between arrows from the left, $\Omega$ functions as a subjective mirror that makes these structural differences explicit from the right, directly at the level of truth-values. 

\end{rmk}

\begin{rmk}
For a topos $\Es$ that satisfies (SG) and thus deserves to be called atomistic, any internal poset $\mathbf{P}$ or topology $j$, both the presheaf topos $\Es^{\mathbf{P}}$ and the sheaf topos $\Sh_j(\Es)$ also satisfy the axiom (SG)\footnote{Ibid., p. 146.}. This is to say that sheafification alone, as Badiou's work on $T$-sets demonstrates, does not nullify atomism or the localic character of a topos unless torsion is explicitly induced at the level of transcendental indexing.
\end{rmk}

\begin{rmk}
If the topos were further to satisfy the axiom of choice (AC)---i.e., that every epimorphism splits---it would not only support generators but the terminal object $1$, which is the maximal element of $\Omega$, would suffice to generate the topos alone. This, together with the postulate of atomism, would make the topos essentially similar to sets.  

Instead, for the general category of localic topoi, Johnstone further distinguishes a slightly weaker condition than the axiom of choice: a topos that \textit{supports splits} (SS). If a topos is both Boolean and satisfies (SS), this mathematically guarantees the postulate of atomism (SG). In this case, not all epimorphisms necessarily split, but for every subobject $U \hookrightarrow 1$, any epimorphism onto $U$ splits, which makes it a projective object in that category. These projective subobjects serve as excellent local representatives that help to test and distinguish between all the arrows or relations---like the very heralds of parliamentary democracy, not to mention the atomizing referee process binding scientific truth---whereas in a general localic topos its structural `atomism' is not as directly visible at an elementwise level. 
\end{rmk}

Based on these two distinct principles---the weak postulate of materialism and the postulate of atomism---it is now possible to isolate the general topos-theoretic conditions governing Badiou's original `postulate of materialism', which should more accurately be termed the \textit{postulate of atomistic materialism}.

\begin{pos}[Strong Postulate of Materialism] \label{strong_postulate}
In addition to the weak postulate of materialism, the \textit{strong postulate of materialism} dictates that a Grothendieck topos defined over $\sets$ must additionally satisfy the postulate of atomism---meaning that its objects are generated strictly by the subobjects of the terminal object $1$. This constitutes the formal definition of atomistic (or democratic) materialism.
\end{pos}

According to the following proposition, the topoi satisfying this strong postulate are shown to be precisely \textit{localic topoi} over $\sets$ and \textit{vice versa}. Namely, for any such topos, one isolates the transcendental lattice $T = \gamma_*(\Omega)$ via its morphism of materialist appearance $\gamma$, and the structural equivalence follows: 

\begin{prop}
If an elementary topos $\Es$ is a Grothendieck topos defined over the base category via $\gamma: \Es \to \sets$, this geometric morphism is bounded if and only if the topos satisfies (SG), and hence the structure of $\Omega$ is preserved by the external Heyting algebra $\gamma_*{\Omega}$, which is a set. 
\end{prop}

\begin{rmk}
In general topos theory, the subobjects of the internal subobject classifier $\Omega$ form a class of arrows $\Hom(\Omega, \Omega)$, just as the global elements of $\Omega$---i.e., the morphisms $1 \to \Omega$---correspond bijectively to the subobjects of $1$. For any staging object $U$, there exist concurrently the subobjects $U^*\Omega$ corresponding to arrows $U \to \Omega$; they represent \textit{generalized points} of the topos---i.e., local minority movements that do not yet communicate at the layer of global, `democratic' screening. This layer of existence is obviously invisible to the atomistic gaze of Badiou's worlds, whose very aim it is to strip democracy of the flourish of pluralism. Therefore, the designation of $T$ as a set of global points remains possible \textit{only} for materialistic (i.e., Grothendieck) topoi that are atomistic (i.e., support generators).
\end{rmk}

\begin{rmk}
Obviously, if one operates within an elementary topos $\Es$ that is not defined over the base category $\sets$ (contrary to the case of Grothendieck topoi), nothing guarantees that atomism (i.e., the axiom SG) would imply the topos to be a locale. This implication holds true only when $\Es$ is a Grothendieck topos  and thus at least weakly material to begin with. The axiom SG, or atomism, could thus be taken as an alternative, incomplete version of Badiou's original postulate, whereas the Grothendieck condition, or materialism, combined with it renders its atomism to be material or constitutive in Badiou's sense. In other words, the two versions of the postulate can be treated in a hierarchical fashion in both directions, with the original, strong postulate containing both. All in all, one can state that the strong postulate of materialism is equivalent to SG (only) \textit{when} the weak condition already applies. 
\end{rmk}

\section{Mathematical Frontier of Dialectical Materialism?}

Where `atomistic' materialism captures many interesting connotations with the invisibility of local, minoritarian flows under Badiou's democracy-concept, I shall now instead focus on the weak postulate of materialism as a condition prone to reflecting the structure of Badiou's own, `dialectical materialist' reasoning. Obviously, his early work implicitly presumed such dialectical structures to reside entirely outside the scope of formal expression. But it is precisely the structural distinction between the two levels of materialism---which eventually compile into a formal articulation of Badiou's `singular' concept of change like the second article\footnote{Veilahti, \textit{Bridging Badiou's Gaps}, 2026.} formulates---that turns central to the passage from the `democratic' to the `dialectical' moment of philosophy. This passage obviously remains invisible to Badiou's own, `atomistic' concept of `transcendental gaze': a localic topos, which constitutes Badiou's formal agora where the strong postulate applies, remains `democratic' in the very sense that there exists only one dominant and thus totalizing general will ($T$) materialized within sets. Or, as Badiou states, it structures a universe where `existence = individual = body' is understandable only in a form reducible to its representation through their global elements. 

This configuration is, of course, fundamentally distinct from Deleuze's concept of multiplicity, even if Badiou explicitly invokes Deleuze within this structural register. Indeed, as opposed to post-structuralist ideals, a localic topos allows for only one objective body of truth, whereas a general Grothendieck topos---let alone a higher topos---invites a process opening up towards non-local interactions of multiple, co-existing bodies, thus questioning any singular definition of coherence (i.e., its grounding on the notion of internal consistency). 

An atomistic world thus completely prohibits subjective torsion, as opposed to a weakly materialist topos where the `fundamental law of the subject' is already (also formally) at risk. Precisely by permitting  torsion at the level of world's `transcendental' (read material) indexing itself, a weakly material (read Grothendieck) topos gives rise to an alternative formulation of truth that offers a genuinely novel interpretation of Badiou's `materialist dialectic'.

Thus far, I have postponed the formal definition of a Grothendieck topos, even if such structures were assumed while phrasing the weak postulate. There exist two distinct ways---the set-theoretically material and the categorically abstract---to introduce this conceptual apparatus. The first definition largely follows the pathway of Badiou's own reasoning: as a world of sheaves  formalized in relation to the three operations Badiou himself distills: localization, compatibility, and amalgamation. 

In the classical case of an ordinary topological space $X$, a sheaf is simply a contravariant presheaf $\F \in \sets^{\Ou(X)}$ satisfying the standard gluing condition. This definition of a sheaf applicable to a locale $\Ou(X)$ exhibits, however, a critical limitation. Rather than operating with an externally given poset of localizations $\Ou(X)$ fixed once and for all, a vigorous materialist theory must consider a multiplicity of such posets themselves---rather than a single, `atomistic' poset $T$---which remain mutually incompatible on a global scale due to internal torsion (read endomorphisms).

Grothendieck followed this precise insight while developing étale cohomology: he did not reject the structural presence of torsion. Yet he nonetheless managed to treat algebraic varieties in a sufficiently `material' manner translatable directly to classical set theory. This operation allowed resuscitating a more  traditional outlook on algebraic geometry, employing hierarchical structures of domination (i.e., sieves) though now applying them as relational contexts rather than as incorporeal substrate inscribed onto the studied manifolds themselves. In other words, Grothendieck topologies provide localized structural `snapshots' that are internally constrained and hierarchical, without conditioning such structures as an internal reality of objects. 
This allows local sites to be evaluated logically without prohibiting torsion on a global scale. 

In general elementary topos theory, the entire internal notion of topology is subsequently generalized, whereas Grothendieck's work on étale theory still requires a set-theoretically material definition of coverage. This formulation precisely allows the `immaterial'---internal or subjective torsion---to \textit{materialize} in a manner that is not merely category-theoretic abstraction, but structurally consequential even in Badiou's incorporeal world of mathematical ontology (i.e. set theory). 

To formalize this trade-off between algebra and geometry, Grothendieck thus resolved the problem of torsion by means that are hierarchical not at the level of ontological contents but instead at the level of the relational, filtering \textit{procedure}: exactly the subjective aspect of mathematical inquiry. Each such hierarchical correlate constitutes a \textit{sieve}, while topology refers to their relational (rather than ontological) synthesis. Namely, a sieve captures a structure that concerns coverings of the form $U_i \to X$: these relational `neighborhoods' themselves can be non-monic, whereas hierarchies or filters are imposed at the descriptive level through their indexing. This means that their compatibility with Badiou's strong postulate remains a `local' effect of relational expression, rather than a matter of global, incorporeal designation. 

To contrast these structures with Badiou's reduced setting, whose modes of existence are still dominated by an \textit{a priori}, atomistic hierarchy, this indexing operates synthetically and retroactively, combining multiple, mutually incompatible hierarchies. It compiles a geometric expression irreducible to a single,  incorporeal filter of contexts.

\begin{dfn}
A \textit{Grothendieck topology} $J$ on a small category $\ce$ associates with each object $C \in \ce$ a \textit{collection} of sieves  $J(C)$, that is, of downward-closed covering families on $C$. 
\end{dfn}

\begin{rmk}
If a morphism $U \to X$ belongs to a sieve $\Si \in J(C)$, then any composite arrow $V \to U \to X$ must likewise reside within that sieve. Thus, the categorical association law renders any sieve a right ideal. In general, a sieve structures a subobject $\Si \hookrightarrow \mathbf{y}(C)$ of the representable presheaf $\mathbf{y}(C) = \Hom(\cdot, C)$ within the presheaf topos $\sets^{\ce^{\op}}$.
\end{rmk}

\begin{dfn}[Closed Sieve] 
A sieve $\Si$ on an object $C$ is \textit{closed} with respect to a topology $J$ if and only if for all morphisms $f: D \to C$, the pullback sieve $f^*\Si \in J(D)$ implies that $f \in \Si$. This, in turn, ensures that $f^*\Si$ is a maximal sieve. 
\end{dfn}

\begin{rmk}
Similar relational configurations are employed implicitly within Paul J. Cohen's forcing procedure, and consequently within \BE. As my second article explores, the so-called `correct sets' $\female \subset \mathbf{P}$ within a forcing poset are closed in a manner analogous to these sieves: whenever an element $p \in \female$ and a condition $q \leq p$ holds, then $q \in \female$. If such a multiplicity $\female$ is `generic'---meaning it intersects every dense subset or domination---it covers the topology implicit in Cohen's argument.
\end{rmk}

\begin{dfn}[Subobject Classifier] 
For a Grothendieck site $(\ce, J)$, the internal subobject classifier $\Omega$---or what I previously labeled as world's \textit{internal} `transcendental' indexing---is defined as the sheaf of closed sieves:
$$ \Omega(C) = \{M \mid M \textrm{ is a closed sieve on } C\}. $$
It compiles a sheaf because the closure-property is compatible with the localized amalgamization (gluing) operation. Within \BE, the subobject classifier can be conceptualized as the formal collection of generic filters relative to a transitive denumerable model $S$.
\end{dfn}

\begin{rmk}
In fact, given any arbitrary sieve $\Si$ on an object $C$, its closure can be formalized explicitly as:\footnote{See Mac Lane and Moerdijk, \textit{Sheaves in Geometry and Logic}, 1992, p. 141, Lemma 1.}
$$ \bar{\Si} = \{h \mid \cod(h) = C \textrm{ and } \Si \textrm{ covers } h\}. $$
This closure operation remains foundational to the categorical definition of topology. In general, a topology is defined relative to the `transcendental' classifier $\Omega$ as an endomorphism $j: \Omega \to \Omega$ satisfying the structural identities $j^2 = j$, $j \circ \true = \true$, and $j \circ \wedge = \wedge \circ (j \times j)$, where the morphism $\true: 1 \to \Omega$ is the generic element corresponding to the transcendental indexing $\Omega$ as a whole. 
\end{rmk}

Within \LW, Badiou implicitly defines a Grothendieck topology when he demonstrates the existence of his `transcendental functor' as part of the `subtle scholium' (Section \ref{tsetisasheaf})---in fact, an implicit formulation of a Grothendieck topology is already effectuated when Badiou formally reconstructs Cohen's forcing apparatus within \BE. Badiou thus defines the `\textit{territory} of an element $p$' explicitly as the collection:\footnote{See Badiou, \LW, 2009, p. 291.}
$$ K(p) = \{\Theta \mid \Theta \subset T \textrm{ and } p = \Sigma \Theta\}. $$
These families form a basis for a Grothendieck topology, while every individual territory generates a  sieve when completed under the ordering relation $\leq$ inherited from the lattice $T$.

\begin{rmk}
To formalize the compatibility conditions that Badiou's proof necessitates, I define these operations directly in terms of a Grothendieck topology. A sieve $\Si$ is defined as a \textit{cover} of an object $C$ if $\Si \in J(C)$. Let $$P: \ce^{\op} \to \sets$$ be an arbitrary presheaf. Then a \textit{matching family} for $\Si$ of elements of $P$ is an assignment which, for each morphism $f: D \to C$ in $\Si$, yields a local section $x_f \in P(D)$ such that $x_f \cdot g = P(g)(x_f) = x_{f \circ g}$ holds for all arrows $g: E \to D$. 

The \textit{amalgamation} of such a matching family is a unique global section $x \in P(C)$ for which $x \cdot f = P(f)(x) = x_f$. Within Badiou's idiosyncratic vocabulary, such matching families are conceptualized as pairwise compatible subsets of an object. An amalgamation element materializing an atom---i.e., an element that in Badiou's vocabulary makes an atom `real'---is not guaranteed to exist in a general setting; the presheaf $P$ is a proper sheaf precisely when such an amalgamation element exists uniquely for all objects $C \in \ce$ and all covering sieves $\Si \in J(C)$. 

In the specific example Badiou himself elaborates, a matching family $\{x_f \mid f \in \Theta \subset T\}$ corresponds directly to the internal atom:
$$ \pi(a) = \Sigma \{\Id(a, x_f) \mid f \in \Theta\}, $$
and this is precisely the structural identity I established in the preceding proof of the subtle scholium (i.e., the sheaf condition\footnote{A presheaf $P$ constitutes a sheaf if and only if for every covering sieve $\Si$ on $C$, the canonical inclusion $\Si \hookrightarrow \mathbf{y}(C)$ into the representable functor induces a bijection $\Hom(\mathbf{y}(C), P) \cong \Hom(S, P)$. In diagrammatic terms, a presheaf $P$ is a sheaf for a topology $J$ if and only if for any coverage basis family $\{f_i: C_i \to C \mid i \in I\} \in K(C)$ in the basis $K$, the following diagram forms an equalizer of sets: $$\xymatrix{P(C) \ar[r] & \prod_{i} P(C_i) \ar@<0.5ex>[r] \ar@<-0.5ex>[r] & \prod_{i, j \in I} P(C_i \times_C C_j)}$$ where the product on the right is evaluated over the fiber product pullback (see Mac Lane and Moerdijk, \textit{Sheaves in Geometry and Logic}, 1992, p. 123). Furthermore, a topology $J$ is defined as \textit{subcanonical} if all representable presheaves $\mathbf{y}(C)$ operate as proper sheaves on $J$.}).

The Grothendieck topology generated by the basis of `territories' on $p$ is subcanonical within the category derived from the poset lattice $T$, where the morphisms are given by the partial order relations. This is because for any element $p \in T$, the hom-set structures $$\Hom_T(q, p) = \{f_{qp} \mid q \leq p\},$$ and therefore the representable functor $h_p = \{q \in T \mid q \leq p\}$, become canonically indexed by the lattice $T$. This mapping renders $h_p$ as bijective to a principal down-set of $T$.

Now, for any territory basis $\Theta$ on $p$, a matching family of such representable subobjects $h_q$ operates as a standard family of restrictions of the global domain $h_p$. In a general setting, for any arbitrary presheaf $P$, I can construct a canonical representation which is subsequently completed as a  sheaf via the sheafification functor, ensuring the unique existence of internal atoms equivalent to the condition postulated within \LW. Presheaves and sheaves can concurrently substitute for one another within specified foundational calculations.
\end{rmk}

\section{Topos: A Categorical Version}

Previous sections reviewed the formulation of structures conditioning a Grothendieck topoi, whereas I shall now focus upon the more abstract, \textit{elementary} definition of a topos to evaluate the significance of weak materiality and what connotations it bears with dialectical materialism in a genuinely categorical setting. The materiality of being can, indeed, be defined in a structural manner that makes no prior material reference to the `ontological' baseline category of sets. 

In the context of this definition, the stakes between being and materiality are thus inverted. From this point of view, a Grothendieck topos is not fundamentally one of sheaves over sets; instead, it is a topos which is not constructed based on a specific, privileged geometric morphism $\Es \to \sets$---i.e., a materialization morphism seeking to transform its geometry into the `Platonic' domain---but rather one for which such a materialization morphism emerges only retroactively, as an effect of relational expression and a product of representation or substrate. This particularly avoids any need to start with set-theoretic constructs like sieves or Badiou's generic filters. At the same time, the strong postulate receives a clear categorical version:

\begin{dfn}[Postulate of Materialism, Categorical Version] 
For a given materialization morphism $\Es \xrightarrow{\ \gamma \ } \Ss$, where the latter is a Grothendieck topos that supports generators and the axiom of choice (thus behaving structurally like sets), the topos $\Es$ is \textit{weakly materialist} if and only if $\gamma$ is bounded as a geometric morphism, which is equivalent to $\Es \isom \Sh(\ce, J)$ for some internal site $(\ce, J)$ in $\Ss$. Moreover, it is \textit{strongly materialist} (atomistic) if and only if the materialization is  \textit{localic} as a geometric morphism, that is, bounded and generated over $\Ss$ by subterminal objects (SG). Equivalently, $\Es \isom \Sh(T, J)$ for the canonical topology of $T = \gamma_*{\Omega}$.
\end{dfn}

This definition obviously presupposes a fully categorical description of a topos, which I shall introduce shortly. The categorical definition is significant because it allows one to identify the precise relational conditions that ensure the elicitability of the materialization morphism in the first place, without attributing materiality as an inherent ontological trait of the underlying topos itself. Rather, materiality is a potential actualized by the (non-necessary) conditions of possibility made expressible by categorical topos theory. This alternative definition also captures the fact that \textit{materialism itself is not a logical question}, unless combined with atomism\footnote{More precisely, when working over the $2$-category of elementary topoi $\Top$, if $\gamma: \Es \to \Fs$ is a geometric morphism which is bounded, then by Diaconescu's theorem there exists an internal category $\mathbf{C}$ in $\Fs$ such that $\Es$ is a reflective subcategory of internal presheaves $\Fs^{\mathbf{C}^{\op}}$ with left exact reflector. Unless the geometric morphism $\gamma: \Es \to \sets$ is bounded, making $\Es$ weakly materialist, the push-forward $\gamma^*(\Omega_{\Es})$ would not need to constitute a complete Heyting algebra. In the localic, strong case, $\Es$ further becomes equivalent to  $\Sh(\gamma_*(\Omega_{\Es}))$; the subterminal objects of $\Es$ are then the opens of this locale. See Johnstone, \textit{Topos Theory}, 1977, p. 150.}.

Therefore, to acknowledge the full significance of the categorical version of the postulate, let me investigate the question of materiality from a more abstract point of view: what constitute the topoi, or `places' of reason, that are not necessarily material to begin with. On the expression of what relational assumptions they become conditioned as possibly so---assumptions of materiality that might differ fundamentally from those defined against the Platonic world of $\sets$? Can one deploy the question of materiality without making any reference---direct or sheaf-theoretic---to the question of what the objects `consist of'? Eventually, can one think about materiality without `crossing' Kant's categorical limit? 
Elementary topos theory demonstrates that one can.

\begin{dfn}[Elementary Topos] \label{toposdefinition}
An \textit{elementary topos} $\Es$ is a category which satisfies three structural criteria:
\begin{enumerate} 
\item(Compatibility with finitude) $\Es$ possesses all finite limits, or equivalently, $\Es$ has finite products, pullbacks, and a terminal object $1$\footnote{A terminal object $1$ means that for every object $X \in \Es$, there exists a unique morphism $X \to 1$.}.
\item (Cartesian closedness) For each object $X \in \Es$, the cartesian product functor $(-) \times X: \Es \to \Es$ possesses a right adjoint, traditionally denoted as the exponential functor $(-)^X: \Es \to \Es$\footnote{Generally, a functor between categories $F: \Fs \to \Gs$ maps each object $X \in \Fs$ to an object $F(X) \in \Gs$ and each morphism $a: X \to Y$ to a morphism $F(a): F(X) \to F(Y)$, satisfying $F(a \circ b) = F(a) \circ F(b)$. A functor $G: \Gs \to \Fs$ is defined as the right adjoint of $F$, denoted by $F \dashv G$, if for any pair of objects $X \in \Fs$ and $Y \in \Gs$ there exists a natural isomorphism $\Hom_{\Gs}(F(X), Y) \cong \Hom_{\Fs}(X, G(Y))$.}.
\item (Internal transcendence) $\Es$ retains an object called the \textit{subobject classifier} $\Omega$, equipped with a universal monomorphism $\true: 1 \to \Omega$ such that for each monomorphism $\sigma: Y \hookrightarrow X$ in $\Es$, there exists a unique classifying morphism $\phi_{\sigma}: X \to \Omega$ making the following square a pullback diagram\footnote{In this specific context, the pullback condition means that given any morphism $\theta: Z \to X$ such that $\phi_{\sigma} \circ \theta = \true \circ !_{Z}$, there exists a unique factorizing morphism $\gamma: Z \to Y$ satisfying $\sigma \circ \gamma = \theta$.}:
$$\xymatrix{
Y \ar@{^(->}[r]^{\sigma} \ar[d] & X \in \Es \ar[d]^{\phi_{\sigma}} \\
1 \ar[r]^{\true} & \Omega.
}$$
\end{enumerate}
\end{dfn}
\begin{rmk}[A Categorical Definition of Materiality] \label{grothendiecktopos}
In contrast to its set-theoretic definition (Definition~\ref{sieve}, p.~\pageref{sieve}), but instead to express the purely categorical conditions of `materiality', a Grothen\-dieck topos can be defined as an elementary topos that satisfies the optional axioms formulated by Jean Giraud:\footnote{See Johnstone, \textit{Topos Theory}, 1977, pp. 16--17.} 
\begin{enumerate}
\item $\Es$ possesses all set-indexed coproducts, and they are disjoint and stable under pullback,
\item equivalence relations within $\Es$ possess universal coequalizers,
\item every equivalence relation in $\Es$ is effective, and every epimorphism in $\Es$ is a coequalizer,
\item $\Es$ has small hom-sets, meaning that for any two objects $X$ and $Y$, the morphisms are parameterized by a classical set, and finally
\item $\Es$ possesses a set of generators (which are not required to be monic with respect to $1$, as is mandatory within the domain of locales).
\end{enumerate}
Together, these five Giraud axioms provide a genuinely categorical definition of a Grothendieck topos, and condition an alternative version of the weak postulate of materialism (Postulate \ref{weakpostulate}, p.~\pageref{weakpostulate}).
\end{rmk}

To bridge Badiou's own, `ontological' definition of objects with this properly categorical framing, I should subsequently demonstrate that Badiou's world of $T$-sets is equivalent to a suitable category of sheaves $\Sh(T, J)$, and that it will, consequently, satisfy the foundational conditions of an elementary topos delineated above. This will be achieved in Section \ref{corrected_proof}. However, to illustrate the shift in the conceptual landscape reverting the hierarchy between relationality and materiality, I briefly comment on how the `Platonic' constructs map onto this fully functorial domain of existence in the first place (i.e. how they exist as categories, before conditioning them as topoi. Indeed, to define an arrow between objects in Badiou's world, it corresponds to the natural transformation encountered through the Yoneda lemma---a functorial relationship that Badiou's framework terms a `complex arrangement'.

\begin{dfn}[Arrow] 
An \textit{arrow} from an object $(A, \Id_{\alpha})$ to an object $(B, \Id_{\beta})$ within the category of $T$-sets is a mapping $\rho: A \to B$ such that it satisfies the two following identities:
$$ \Ee_{\beta} \rho(a) = \Ee_{\alpha} a  \quad \text{ and } \quad \rho(a \res {p}) = \rho(a) \res {p}. $$ 
\end{dfn}

\begin{prop}
It is a straightforward consequence of these two presuppositions that the map $\rho$ strictly preserves the underlying order relation $\leq$, yielding the inequality:
$$ \Id_{\alpha}(a, b) \leq \Id_{\beta}(\rho(a), \rho(b)) $$
for all elements in the domain. Furthermore, if $a \dagger b$ are a pair of compatible elements, their images likewise remain compatible, satisfying the relation $\rho(a) \dagger \rho(b)$ and thus making the map $\rho$ compatible with the underlying $T$-structures\footnote{See Badiou, \LW, 2009, pp. 338--339.}.
\end{prop}

Given categorical definitions, and in spite of Badiou's\footnote{Ibid., p. 339.} structural confusion regarding power-objects, one may assert that Badiou presents at least a valid \textit{category} of $T$-sets. Its objects are $T$-sets in the world $\m$\footnote{\label{setsworld}Despite the fact that the category of sets, $\sets$---which is not itself a set, as Badiou demonstrates in his treatment of the world $\m$---is paradoxical by virtue of Russell's argument, one may safely replace $\sets$ with any small-generated base topos $\Ss$ satisfying the categorical version of the Axiom of Choice (AC), supporting a set of generators (SG), possessing a natural numbers object, and whose subobject classifier equals the bi-valued lattice $\Two$.}, with evaluation identity functions $\Id_{\alpha}$. 

However, exemplifying the limitations inherent to Badiou privileging set-theoretic language, as  Madarasz\footnote{Madarasz,  `On Alain Badiou's Treatment of Category Theory in View of a Transitory Ontology', 2005. } observed and which became increasingly evident through \LW, Badiou was led to assume that every category should take a similar localic form---a classical fallacy known as affirming the consequent, which is both technically and philosophically alarming.

\section{Badiou's Diagrammatic Struggle}

Risking confusion between general topoi with localic extensions of modal logic, Badiou believes to have overcome the need for `dia\-gram\-matic' operations\bk ---as if everything, despite notational discrepancies, could be reconciled with `Platonic' foundations, bringing categorical reasoning back down to classical set theory. By his manoeuvre, Badiou inevitably strikes against how `diagrams' are utilized not only by modern mathematicians but by several post-structuralist thinkers, especially since the term was first launched by Michel Foucault\footnote{See Foucault, Michel (1995), \textit{Discipline and Punish: The Birth of the Prison}, trans. Alan Sheridan. New York: Random House. pp. 171, 205.}. Obviously, even if this notational reversal were attainable, the translation of topos theory into the reduced domain of conjunction-semilattices would hardly serve as evidence. What evidence rather points at is how Badiou himself bases his claims on mastery that is hardly convincing, despite what might superficially appear. Indeed, not only does Badiou commit a converse error, but he struggles even with the formal elaborations concerning his own `matheme', whose portrayal reveals significant limitations of  expertise. 

For example, within the category of $T$-sets, he attempts to demonstrate the existence of an object which he calls an `exponent' of a relation $X \to Y$. In his narrative register: `[i]f we consider [\dots] two Quebecois in turn as objects of the world, we see that they each entertain a relation to each of the objects linked together by the `Oka incident' [\dots] and, by the same token, a relation to this link itself, that is a relation to a relation'\footnote{Badiou, \LW, 2009, p. 313.}. Following this line of reasoning, Badiou defines a relation to be `universally exposed' if, given two distinct expositions of the same relation, there exists between the two exponents one and only one relation such that the following diagram remains commutative\footnote{Ibid., p. 317.}: 
$$\xymatrix{
& \textrm{Exponent 1} \ar[ddl] \ar[ddrrr] & & \textrm{Exponent 2} \ar[ddlll] \ar[ll]_{\textrm{unique relation}} \ar[ddr] & \\
&&&&\\
\textrm{Object 1} \ar[rrrr]^{\textrm{Exposed relation}} &&&& \textrm{Object 2}.
}$$ 
He exemplifies this by a `progressive Quebecois citizen' (Exponent 1), coexposing the same relation together with another, `reactionary citizen' (Exponent 2): `if the reactionary supports the progressive, who in turn supports the Mohawks, there follows a flagrant contradiction with the direct relation of vituperation that the reactionary entertains with the Mohawks'\footnote{Ibid., pp. 315--316.}. Yet, the conclusion this narrative serves to logically support is unwarranted: Badiou falsely argues for the necessity of the non-commutativity of this diagram, in spite of its commutativity remaining perfectly possible in purely mathematical sense, thereby confusing the necessary with the sufficient.

If this is still circumstantial, the evidence for Badiou's diagrammatic stumbling becomes more tangible as he formally engages properties like universal exposition:  

\begin{prop}
Badiou's formalized condition of universal exposition is falsely stated. Given any arbitrary relation $f: X \to Y$, it is straightforward to construct an object that fails to satisfy Badiou's universality condition within almost any base topos, and explicitly within the category of sets. 
\end{prop}

\begin{proof}
To simplify, I shall evaluate this limitation in the category of $\sets$, where the transcendental lattice $T = \Two$ has two elements and where relations (functions) can be specified pointwise. I construct a diagram where $\pi_1, \pi_2$ denote the two canonical projections of the Cartesian product $X \times X \times X$:
$$\xymatrix{X \times X \times X  \ar[dd]^{\pi_1} \ar[ddrr]_{\pi_2 \qquad} \ar@{-->}[rr]^{h [?]} & & X \times X \times X \ar[ddll]^{\qquad f \circ \pi_1 } \ar[dd]^{f \circ\pi_2} \\ & & \\
X \ar[rr]^f & & Y.}$$

It remains clear that the third component of the induced map $h$---namely the composition $X \xrightarrow{\;\iota_3\;} X^3 \xrightarrow{\;h\;} X^3 \xrightarrow{ \;\pi_3\;} X$---can be defined arbitrarily because none of the composite morphisms $f \circ \pi_1$ and $f\circ\pi_2$ operating on the domain $X \times X \times X$ refer to the third projection component. The only structural restriction concerns the first component of $h$, which must satisfy specified compatibility conditions with respect to the second component within the image of $h$. 
In other words, under Badiou's restricted definition, a relation is never `universally' exposed (read uniquely) as long as the category provides sufficient room to play, as is the case within the classical category of sets. \end{proof}

There exists a further terminological confusion that concerns the way Badiou assumes an object to expose a \textit{single} relation. As such, his `exponent' stands in direct contradiction to the standard definition native to category theory proper, which formalizes the exponential object $X^Y$ as the universe of \textit{all} internal morphisms between $X$ and $Y$, rather than exposing a single isolated relation $f: X \to Y$.

What Badiou formally exposes is, in fact, a \textit{graph} of the given relation $f: X \to Y$ in conventional mathematical terminology. The correct form of exposition---should it follow the  intuitions deployed within his Quebecois example---would dictate that whenever another object $Z$ concurrently exposes that relation $f$, there exists a unique factorization arrow $Z \to \Gamma_f$ that leads to a commutative diagram 
$$\xymatrix{Z  \ar[dd]^{} \ar[ddrr]_{} \ar[rr]^{} & & \Gamma_f \ar[ddll]^{ } \ar[dd]^{} \\ & & \\
X \ar[rr]^f & & Y.}$$
Even with this clarification of what is practically exposed, the central aspect of the desired universality claim is to specify this object $\Gamma_f$ up to a unique isomorphism---i.e., that it is structurally unique. Reading Badiou alone hardly communicates the fact that the universal exposition is not a property of an isolated arrow, but of its existence in the entire relational context of the category, which the graph uniquely embodies. It is this categorical condition, rather than the graph itself, that constitutes a \textit{universal property}. Within Badiou's example, this property defines not an arbitrary citizen, but a unique, `universal' citizen associated with the specific relation of the Oka incident---call it $\Gamma_{\textrm{Oka incident}}$. Or, thinking about how subjectivity emerges through points, they are still `generic' to Badiou, rather than claims to universality (which they structurally pursue through his argument). 

\begin{rmk}
This inability to define what precisely constitutes the notion of `universality', or rather a claim to universality as a specific categorical condition, demonstrates Badiou's at best partial comprehension of the categorical mode of reasoning. He fails to execute that precise Kantian shift that declines to specify objects by their contents. 

Such universality claims remain foundational to the way of thinking typical across various branches of modern mathematics, already before category theory had explicitly formalized such relational dynamics. Lacking this insight, Badiou could not but denounce the diagrammatic approach in advance, instead of honestly exploring how exactly the mathematicians of today execute their power. Inevitably, there are analogues as mathematicians now ontologize relations rather than contents. But, in light of his diagrammatic struggles, he at least himself failed to provide material evidence to support his conclusion about the `same mathematics' still being in operation, i.e., what precisely is the same and what is different.

To exemplify the power of this rebuilt mode of reasoning, the first condition within the general definition of an elementary topos $\Es$ necessitates the existence of all finite limits\footnote{If a category $\mathcal{J}$ is finite, I consider the functor category $\Es^{\mathcal{J}}$ whose objects are functors $F: \mathcal{J} \to \Es$ and whose morphisms constitute natural transformations. There exists a diagonal functor $\Delta_{\mathcal{J}}: \Es \to \Es^{\mathcal{J}}$ defined as the constant functor assigning $\Delta(X)(j) = X$ for all indices $j \in \mathcal{J}$. The limit condition dictates that there exists an exact right adjoint to the diagonal functor, denoted as the limit functor $\varprojlim_{\mathcal{J}}: \Es^{\mathcal{J}} \to \Es$. An alternative, equivalent definition assumes exclusively the existence of pullbacks and a terminal object $1$, from which the existence of all finite limits rigorously follows.}. This condition is logically equivalent to the simultaneous existence of a terminal object and pullbacks, all defined by universality claims analogous to that of a graph. 
\end{rmk}

\begin{dfn}[Pullback] 
For any given pair of arrows $f: X \to Z$ and $g: Y \to Z$, their pullback is an object $X \times_Z Y$ equipped with projection arrows $p: X \times_Z Y \to X$ and $q: X \times_Z Y \to Y$ such that $f \circ p = g \circ q$. It satisfies the following universal property: given any object $W$ and arrows $h: W \to X$ and $k: W \to Y$ such that $f \circ h = g \circ k$, there exists a unique arrow $u: W \to X \times_Z Y$ making the entire diagram commute:
$$
\xymatrix{
W \ar@/_/[ddr]_k \ar@/^/[drr]^h \ar@{..>}[dr]|-{u} \\
& X \times_Z Y \ar[d]^q \ar[r]_p & X \ar[d]_f \\
& Y \ar[r]^g & Z.
} 
$$
The object $X \times_Z Y$ is uniquely determined up to an isomorphism.
\end{dfn}

\begin{prop}[Exposition of a Single Relation]
The corrected version of Badiou's specific form of `universal exposition' (i.e., of a particular relation)  reduces to the statement that there exists a pullback within the category of $T$-sets, formalized via the following commutative diagram:
\begin{equation} \label{universal_exposure} \xymatrix{
F_{\rho} \cong \Gamma_{\rho:A \to B} \ar@{-->}[d]^g \ar@{-->}[r]^f &
A \ar[d]_{\rho} \\
B \ar@{=}[r]^{1_B} & B}\end{equation}
\end{prop}

\begin{rmk}
With only minor technical complications, Badiou's proof of this proposition could be generalized to demonstrate that \textit{any} arbitrary pullback exists within this universe of discourse. This verification would, in fact, complete the formal proof of the first primary condition required of an elementary topos (Definition~\ref{toposdefinition}). Likewise, other topos-theoretic constructs like finite products, terminal object, and the `transcendental' map $1 \xrightarrow{ \textrm{true} } \Omega$ are defined by analogous universality claims. Their demonstration is heuristically similar to that required by a graph---a proof present in \LW\ with little clarification: 
\end{rmk}

\begin{proof}[Partial proof] 
Badiou phrases the underlying universality condition mistakenly and does not account for its relational uniqueness, as he only demonstrates the \textit{existence} of such a functional graph. The proof of existence proceeds by first demonstrating that the standard Cartesian product exists within his world of $T$-sets: `[g]iven two multiples $A$ and $B$ appearing in a world, the product of these two sets, that is the set constituted by all the ordered pairs of elements of $A$ and $B$ (in this order), must also appear in this world'\footnote{Badiou, \LW, 2009, p. 345.}. The subsequent stage consists of demonstrating, for a given relation $\rho: (A, \Id_{\alpha}) \to (B, \Id_{\beta})$, that the multiplicity consisting of the pairs $(x, \rho(x)) \subset A \times B$, denoted explicitly by $F_{\rho}$, forms a suitable object (i.e., a `real atom')\footnote{Ibid., p. 346, Lemma 2.}.

When equipped with an evaluation mapping $\nu: F_{\rho} \times F_{\rho} \to T$, where:
$$ \nu((a, \rho(a)), (b, \rho(b))) = \Id_{\alpha}(a, b) \wedge \Id_{\beta}(\rho(a), \rho(b)), $$
this mapping proves out to be compatible with the Heyting structure and the axioms of a transcendental indexing:\footnote{Ibid., p. 346, Lemma 3.} 
$$ \nu(x, y) \wedge \nu(y, z) \leq \nu(x, z). $$

To demonstrate that the graph $F_{\rho}$ operates as an object within the category, Badiou is required to verify that every internal atom is real\footnote{Ibid., p. 347, Lemma 4.}. Given an arbitrary atom $\epsilon: F_{\rho} \to T$, Badiou\footnote{Ibid., pp. 347--348.} constructs an induced mapping $\epsilon^*: a \mapsto \epsilon(a, \rho(a))$ from $A \to T$, which is an atom since $\epsilon$ satisfies the required stability conditions. Hence, by virtue of the `postulate of materialism', this atom is verified as real, satisfying the identity $\epsilon^*(x) = \Id_{\alpha}(c, x)$ for a unique element $c \in A$. 

It subsequently follows that $\epsilon(x, \rho(x)) = \nu((c, \rho(c)), (x, \rho(x)))$, which completes the proof of Badiou's Lemma 4. By his subsequent Lemma 5, Badiou demonstrates that the diagram (\ref{universal_exposure}) is valid, asserting that `[t]he object $(F_{\rho}, \nu)$ is an exponent of the relation $\rho$'\footnote{Ibid., p. 349.}---a fact that follows by verifying that the structural projections $f$ and $g$ conserve localizations, that is, $f(a \res {p}) = f(a) \res {p}$ and $g(a \res {p}) = g(a) \res {p}$, while also satisfying $(a, \rho(a)) \res {p} = a \res {p}$. 

By virtue of the definition of the localized segments, it remains straightforward to verify via classical set theory that the object satisfies the existential parameters of the pullback diagram (\ref{universal_exposure}, p.~\pageref{universal_exposure}); thus Badiou\footnote{Ibid., pp. 350--352.} is only required to demonstrate that the mapping conserves existence metrics and localizations. This completes the existence-part of the proof while the relational uniqueness remains to be shown.
\end{proof}

\section{$T$-sets Form a Topos: the True Scholium with Proof} \label{corrected_proof}

I now close the discussion concerning the `logical completeness of [Badiou's] world'\footnote{Ibid., p. 318.} by demonstrating that his world of $T$-sets does indeed give rise to an elementary topos, satisfying the three principal conditions (Definition \ref{toposdefinition}).

\begin{thm} \label{finalscholium}

Badiou's world consisting of $T$-sets---formalized as pairs $(A, \Id)$ where the identity mapping $\Id: A \times A \to T$ satisfies the specific structural conditions---is `logically closed'; that is, it constitutes an \textbf{elementary topos}. It thus encloses not only pullbacks but also the exponential functor. These features ultimately render it possible for this categorical topos to internalize Badiou's transfinite arguments that operate on the power-functor, as illustrated in the second article. They can therefore be expressed from within the situation of a topos, regardless of whether one endorses a fixed, `Platonic' category of sets. 
\end{thm}

\begin{proof} 
Rather than initiating the proof from within Badiou’s localic formalism of $T$-sets, I refer to the standard mathematical literature through which $T$-sets are recognized as sheaves. In other words, when $T$ is viewed as a posetal category via its internal order $\leq$, there exists a canonical equivalence of categories between $T$-sets satisfying Badiou's `postulate of materialism' (i.e., where atoms are `real') and the category of sheaves $\Sh(T, J)$ over the semilattice $T$ equipped with its canonical Grothendieck topology $J$, where the coverage basis $$K(p) = \{\Theta \subset T \mid \Sigma \Theta = p\}$$ is generated by the atoms or `territories' on $p$ (as they are called in \LW). The technical complications Badiou is caught up with while seeking to `Platonize’ the existence of a topos thus largely go in vain; it suffices to demonstrate that $\Sh(T, J)$ is a valid Grothendieck topos, which then automatically satisfies the axioms of an elementary topos.

Consider the associated sheaf functor (the sheafification reflector) that universally exists for the category of presheaves:
$$ a: \sets^{T^{\op}} \to \Sh(T,J) $$
where $\sets^{T^{\op}}$ is the category of presheaves on $T$, and $a$ is the exact left adjoint to the inclusion functor $\Sh(T,J) \hookrightarrow \sets^{T^{\op}}$. Because $a$ is exact (preserving finite limits) and $\Sh(T,J)$ is a reflective subcategory, the exponential objects and limits in the sheaf topos $\Sh(T, J)$ are directly induced by the presheaf category $\sets^{T^{\op}}$, corrected by the closure conditions of the topology $J$.

In the classical base category of sets---assuming that this `Platonic' universe exists---I define $Y^X$ as the set of all functions $X \to Y$. Then, within the category of presheaves $\sets^{T^{\op}}$, the exponential object satisfies the canonical isomorphism:
$$ Y^X(U) \cong \Hom(h_U, Y^X) \cong \Hom(h_U \times X, Y) $$
where $h_U$ denotes the representable presheaf $h_U(V) = \Hom(V, U)$ by virtue of the Yoneda lemma. This exponential adjunction holds for all sheaves rather than merely representable ones, because $\Sh(T,J)$ is exponentially closed in $\sets^{T^{\op}}$. The verification subsequently follows through categorically defined limits, whose existence constitutes a standard task within categorical logic\footnote{See Johnstone, \textit{Topos Theory}, 1977, pp. 24--25.}. It can be verified directly that the presheaf $Y^X$ satisfies the structural criteria of a sheaf when $X$ and $Y$ are sheaves.

Finally, regarding the subobject classifier $\Omega_{\Sh(T, J)}$ within the category of proper sheaves, the assignment is formalized via $J$-closed sieves\footnote{Ibid., p. 25.}:
$$ \Omega_{\Sh(T, J)}(p) = \{\text{closed sieves } \Si \text{ on } p\}. $$

Here, it is critical to remind ourselves that the Grothendieck topology on the semilattice $T$ is defined by the coverage basis of territories. Therefore, in the specific case of $T$-sets satisfying Badiou's strong `postulate of materialism', the subobject classifier $\Omega_T(p)$ consists of all sieves $\Si$ (downward-closed subsets) of $T$ bounded by the structural relation $\Sigma \Si \leq p$. These sieves are further required to be strictly \textit{closed}. 

A sieve $\Si$ with an envelope $\Sigma \Si = s$ is closed if for any other element $r \leq s$---that is, for all arrows $f_{rs}: r \to s$ in the poset category---one establishes the structural implication:
$$ f_{rs}^*(\Si) \in J(r) \quad \Longrightarrow \quad f_{rs} \in \Si. $$

In particular, since $\Sigma \Si = s$ under the topology whose basis consists of territories on $s$, there holds $1_s^*(\Si) = f_{ss}^*(\Si) = \Si \in J(s)$. Now, the exact condition that the sieve is closed forces the identity morphism $1_s \in \Si$. This condition is satisfied exclusively when the sieve $\Si$ operates as the maximal sieve on $s$---namely, it consists of \textit{all} structural arrows $r \to s$ for $r \leq s$. In such a setting, it is straightforward to verify that $\Si$ itself is closed. Therefore, within this restricted localic case, the classifier reduces to:
$$ \Omega_T(p) = \{\downarrow(s) \mid s \leq p\} = \{h_s \mid s \leq p\}. $$
It is no more difficult to verify that this configuration yields a proper sheaf whose amalgamations are  `real' in Badiou's sense, thus retaining a well-defined $T$-structure.

Let me now assume that one is given an object $A$, which is a presheaf functor and thus a $T$-graded family of subsets $A(p)$. For there to exist a structural sub-functor monomorphism $B \hookrightarrow A$ requires that $B(p) \subset A(p)$ for each element $p \in T$. For each partial ordering $q \leq p$, one also possesses a canonical injection $B(q) \hookrightarrow B(p)$ compatible through the subset representation with the injections $A(q) \hookrightarrow A(p)$. 

For any given section $x \in A(p)$, I evaluate the following collection:
$$ \phi_p(x) = \{q \in T \mid q \leq p \text{ and } x \res q \in B(q)\}. $$
This collection is a sieve on $p$ by virtue of the compatibility condition for injections. Furthermore, it is verified as closed because the mapping of the form $x \mapsto \Sigma \phi_p(x)$ constitutes an `atom', to use Badiou's vocabulary, thereby retaining a unique real representative element $b \in B$. It subsequently turns out that the sieve satisfies the structural identity $\phi_p(x) = \hspace{3pt } \downarrow (\Ee b)$. 

There now exists a transformation of functors $\phi: A \to \Omega_T$ which is natural and diagrammatically compatible. As a consequence, the injection (monomorphism) $B \hookrightarrow A$ emerges as the pullback along $\phi$ of the universal morphism $\true$:
$$\xymatrix{
B \ar[r] \ar@{^(->}[d] & 1 \ar[d]^{\true} \\
A \ar[r]^{\phi} & \Omega.
}$$

It remains a straightforward exercise to demonstrate that this mapping satisfies the universal property of a pullback diagram. Indeed, let there exist another natural transformation $\theta: A \to \Omega_T$ making the diagram commute. Given a section $x \in A(p)$ and an arrow $q \to p$, the pullback condition entails that $x \res q \in B(q)$ if and only if $\theta(x \res q) = \true_q$. By the naturality of the transformation $\theta$, this identity is equivalent to $\theta_p(x) \res q = \true_q$, which means that the arrow $(q \to p)$ resides in $\theta_p(x)$. Therefore, the equality $\phi_p(x) = \theta_p(x) $ holds universally. 

This concludes my proof demonstrating that the universal morphism $\true: 1 \hookrightarrow \Omega_T$, which associates to the unique element of $1(p)$ the maximal sieve $\downarrow(p)$, renders the object $\Omega_T$ a subobject classifier of the sheaf category $\Sh(T, J)$, which in turn is equivalent to Badiou's category of $T$-sets.
\end{proof}

\section{Conclusion---Badiou Inside Out}

In the fourth and final book of the `Greater Logic'---the foundational part of \LW---Badiou attempts to conduct the exact proof I sketched in the previous section. In Badiou's\footnote{Badiou, \LW, 2009, p. 299.} own vocabulary, this demonstration constitutes the `analytic' of the inquiry, `bearing only on the transcendental laws of being-there', or in other words `the theory of worlds, the elucidation of the most abstract laws of that which constitutes a world qua general form of appearing'. However, Badiou fails to execute the necessary categorical verifications required to support the claim for its `generality'---a questionable claim that, he symptomatically believes, should denounce the relevance of categorical being-there. As a result, Badiou affirms the consequent immersed within his own atomistic framework, adopting a \textit{localically bounded}, reductive approach to topos theory.\footnote{I have demonstrated how his formalism reduces topos theory to a theory of locales---elementary topoi materialized over the base category of $\sets$ in such a manner that the materialization geometric morphism $\gamma: \Es \to \sets$ is bounded and localic in the standard vocabulary of geometric morphisms. This restriction leads Badiou to identify the internal logic of a topos with the external logic emerging on the set-theoretic surface of $T := \gamma_*(\Omega)$, whereas for a general Grothendieck topos, these two algebras do not coincide.}

If Badiou's dialectical apparatus subsequently aims to bridge this `analytics' with the meta-ontology of \BE, the term `analytic' refers precisely to the strong, bounding condition of the constitutive, atomistic postulate of materialism. Similar limitations do not apply to `weak materialism'; rather, it is tempting to interpret the materially weak Grothendieck topoi as mathematical reconstructions of Badiou's own dialectical materialist philosophy---reconstructions whose existence his `dialectical' program precisely is reluctant to admit. Grothendieck, early on, approached the mathematical \textit{a priori} from a relationally synthetic, `Kantian' perspective (as a prolific variety of sieves and fibers), whereas Badiou imposes a condition of localic boundedness as if its implied, incorporeal hierarchies (like $T$) conditioning the possibility of world's internal transcendence were an \textit{a priori} condition of any possible formal analytic. 

The postulate of materialism thus proves out to \textit{split in itself}---instead of offering a sustained foundation of the necessity of world's `dialectical' splitting. It especially carries a fundamentally different meaning within atomistically (read localically) bounded and non-atomistic situations. In essence, Badiou thus projects the defining, binary topology of the postulate itself onto the world it arguably conditions. Symptomatically, it is exclusively within the \textit{quasi-split}\footnote{The split understanding of truth refers to the specific condition of the category of sets, $\sets$, where the subobject classifier $\Omega$ is the split Boolean lattice consisting of only two elements: true and false. Similarly, although the truth within a general locale does not split in this explicit manner, it remains quasi-split insofar as the lattice $T$ set-theoretically represents truth as extensive `values' or `degrees'. Badiou misleadingly calls them `degrees of intensity', whereas the categorical description of an elementary topos expresses such constraints in a significantly more intensive, geometric meaning than Badiou's extensive, element-wise description.} situation of a localic topos that the two distinct conditions of materiality---i.e., the split topology of the `dialectical' intervention itself---become indistinguishable. This apparent formal convergence captured in Badiou's own postulate, in turn, directly contradicts his other core maxim: if dialectics understands `truths as exceptions', these exceptions must be evaluated in terms of the `most abstract laws'. This is supposedly because the `objective domain of their emergence [\dots] cannot yet attain the comprehension of the terms that singularize materialist dialectic [\dots] and subjects as the active forms of these exceptions'\footnote{Ibid., p. 299.}---a singularization of Badiou's philosophical project that the relational laws of topos theory precisely comprehend. 

Badiou's struggle with diagrammatics thus renders his own comprehension of these abstract laws structurally flawed, precisely because they can---in direct opposition to what Badiou claims---be thought with mathematical rigour. Even if Badiou's own understanding of materialist truth is (quasi-)split and mathematically idiosyncratic---grounded either upon the absolute primacy of the split concept of truth (the Axiom of Choice within \BE) or the quasi-split manifestation of force demonstrated by his fundamental \textit{law of the subject}\footnote{See Badiou, \BE, 2006, p. 401.}---this by no means delimits the general topos of mathematical comprehension. It is Badiou himself who forces the dualistic choice: `mathematics or poem', assuming a splitting between the world's two faces of appearance, instead of explaining them. Therefore, rather than presenting a credible bridge between ontology and phenomenology, what \LW\ bridges to is merely a retreat to his old, absolutist notion of subjectivity founded on linguistic contradictions.  

At the same time, Badiou fails to validate his own dialectical maxim, for there exists no structural reason to exempt mathematics from his rule of exceptions: namely, that `[t]here are bodies and languages, \textit{except} that there are truths'\footnote{Badiou, \LW, 2009, p. 6.}. This dialectical exception should apply to the universe of mathematics like it applies to any other.  In particular, its embodiment within classical set theory remains exceptional not only because the `Platonic' domain bounds its truth logically---as a split, two-valued choice as an archetype of global hierarchy--but also because the `Platonic' domain seems incapable of endorsing the multiplicity of the bodies (and not just languages) of mathematical expression: a fact that topos theory precisely affirms\footnote{Veilahti, `Bridging Badiou's Gaps', 2026.}. 

To Badiou, therefore, there are no `bodies and languages' in the plural---not at least in any mathematical sense of plurality; the title of Badiou's second treatise regarding the multiplicity of `worlds' is conceivable only insofar as the internal `logics' or languages differ while the body of expression, to him, always remains `the same'. The materiality of mathematics appears `democratic' (read `atomistic'), but only from within this bounded, inside-out reading of topos theory.

As a consequence, to Badiou's philosophy there exists but a single `affirmative proposition'\footnote{Ibid., p. 249.} that declares the `existence of a pure form of the multiple'. There exists no intermediate interaction; no Deleuzian line of flight in the middle. The event is structurally prevented from emerging or taking place in the midst of mathematics, nor can it occur in the middle of worlds. Badiou's framework itself stumbles upon its own supposed exceptionality: only a \textit{singular}, formally isolated event whose `existence $=$ individual $=$ body'\footnote{Badiou, \LW, 2009, p. 2.} becomes the highest maxim of dialectical materialism, bordering on existential narcissism. It holds a bounded truth enacted by structures strictly \textit{internal} to an elementary topos, as opposed to the dogmatically externalized topos Badiou never truly manages to step outside.

Were these formal shortcomings accidental, driven by an unfortunate, restricted selection of literature?\footnote{In the mathematical references Badiou provides, every major text relates exclusively to these logically bounded, localic structures. Badiou (ibid., p. 538) cites: Bell, \textit{Toposes and Local Set Theories: An Introduction}, 1988; Borceux, \textit{Handbook of Categorical Algebra: Volume 1}, 1994; Goldblatt, \textit{Topoi: The Categorical Analysis of Logic}, 1984; and Wyler, \textit{Lecture Notes on Topoi and Quasi-Topoi}, 1991.} Was it by pure chance that Badiou encounters phenomenology only in a manner constrained by the intuitionist propositional calculus enacted by Heyting algebras? At least, Badiou himself hardly counts as an intuitionist; within \LW, he explicitly announces that `intuitionism has mistaken the route in trying to apply back onto ontology criteria of connection which \textit{come from elsewhere}, and especially from a doctrine of mentally effective operations', declaring further that `intuitionism is a prisoner of the empiricist and illusory representation of mathematical objects'\footnote{Badiou, \LW, 2009, p. 249.}. Yet, this reduction to localic topoi, bounded by their atomistic description, serves a deeper, strategic philosophical purpose.

Indeed, Badiou\footnote{Ibid., p. 300.} subsequently claims to be able to contribute to several classic philosophical debates. The first problem he discusses has haunted philosophy since Democritus and Lucretius, who `had glimpsed the possibility of an infinite plurality of worlds'. This concerns the structural capacity of transfinite set theory to reflect upon infinity. As Badiou\footnote{Ibid., p. 300.} maintains, to Aristotle the question of the `arrangement of the world [which] is essentially finite' represents a fundamental problem destined to persist `at the heart of the Kantian dialectic'.

But these transfinite arguments can be executed from within any elementary topos---not restricted exclusively to the `ontological' category of sets. There exists, consequently, an infinite number of diagrammatic arrangements to construct an argument concerning infinity. The two distinct meanings of infinity---which in Badiou's terminology translate to `bodies' and `languages'---hardly communicate when evaluated within Badiou's own computational architecture. If Badiou was on the right track while seeking to demonstrate that the logics of worlds are, so to speak, infinite, where he fails is in assuming that only languages, but never bodies expressing them, are unbounded.

The second article of this series extends Badiou's notions of body and change to the categorical framework of topos theory\footnote{See Veilahti, `Bridging Badiou's Gaps', 2026.}. By contrast, any mathematical bodies hosted by Badiou's worlds remain strictly constrained within the domain of the finite, Boolean matrix---an artificial limitation he presumes to support his case that bodies should exist exclusively as something non-mathematical, that is, that mathematics itself would not count as a body. Unsurprisingly, therefore, Badiou's mathematically atomistic program calls for a defensive return to Aristotle.

But there exists also a second structural concern this train of thought illustrates---one rooted deeply within the Kantian critique. Badiou argues that an `object is nothing but the legislation of appearing'\footnote{Badiou, \LW, 2009, p. 301.}. Instead of propagating a relationalist view---as I have done throughout my categorical discussion---Badiou rather claims that `[t]he definition of a relation must be strictly dependent on that of objects'\footnote{Ibid., p. 301.}. This is a claim that, without any recourse to an empiricist conclusion (Badiou's own experience of mathematics), supports my hypothesis that the purpose of Badiou's mistake constitutes something far more systemic than a mere historical accident.

'Wittgenstein', Badiou continues, `having defined the `state of affairs' as a `combination of objects', posits that, `if a thing can occur in a state of affairs, the possibility of the state of affairs must be written into the thing itself''\footnote{Ibid., p. 301.}. Badiou still treats the body or the topos of mathematics as such a `thing': a static object that the `logics of worlds' are written \textit{into}, instead of recognizing the things or objects themselves as they dynamically occur and conquer mathematics---that is, in a relation \textit{for} themselves. Mathematics is still viewed as a science of being understood only through its designation, as one `fixed in its mirage', instead of allowing mathematical expression itself to `happen' and encompass its own shifts and change. The domain of being, in its Badiouian reading, remains incorporeal. It is tempting to argue that this is far from an inadvertent slip, but supports Badiou's ultimately despotic political program against a genuine multiplicity. With ostensive support from his programmatic formalism, he leverages on the obvious shortcomings of democratic representation to reinstitute an absolutist doctrine of Truth. As a result, the language of being is yet to discover the process of its own inscription; its own \textit{bodies of expression}, as I frame it in the very title of the second article\footnote{Veilahti, \textit{Bridging Badiou's Gaps: Topos Theory and the Bodies of Mathematical Expression}, 2026.}.

What constitutes mathematics beyond the `interior of the world' then?\footnote{Badiou, \LW, 2009, p. 262.} How does it inscribe its own exteriority from inside of its own topos to gain a structural grasp of its own alternatives? Badiou could argue that, even if one were shown that there exist not only multiple languages but diverse bodies (topoi) to mathematics, the situations internal to them would still remain mute and bounded by formal languages, as Wittgenstein anticipated. They would still behave `logically', if only from within. This is to say that there would still exist a language, or a `real synthesis' of such languages, that a topos might localize or embody but without truly violating any of Badiou's `democratic materialist' claims about mathematics. 

Even if structurally different from classical set theory, an elementary topos would thus be treated within his `transcendental gaze' as a world where mathematics happens and `communicates' still in largely `ontological' terms---even when shifted to the level of relations. The `same mathematics' would still continue, despite now being expressed in some standalone topos rather than sets. This is, indeed, the case in elementary topos theory---but only when evaluating it exclusively from the point of view of a single, \textit{isolated topos}; something internally singular. Instead, when one evaluates those topoi as they themselves geometrically interact, mutate, and transform, there occurs a structural shift nullifying the ballast of internal boundedness. 

Badiou's mistake thus materializes precisely through this singular per\-spective-error: he cares merely about what happens from inside a single situation\bk ---whether described categorically or element-wise. Again, by an argument grounded upon the internal conditions of such a standalone topos, he affirms the consequent by suggesting that no other way for situations to exist, except through their evental collapse,  remains genuinely possible: a formal, but unwarranted restriction mirrored in his philosophical reflection of subjective inwardness---the philosophical `leitmotiv' he himself recognizes as running continuously from Descartes through Hume to Hegel\footnote{See Badiou, \LW, 2009, p. 411.}. As the second article points out, he completely overlooks how multiple topoi relate structurally to each other through non-local, geometric morphisms.

This possibility of multiple bodies, or worlds, opens an entirely new frontier where Badiou's isolated image of being cannot be saved even as an internal structural condition, not even when reading his theory of worlds `inside-out'. As I explore in another article\footnote{Veilahti, Antti (2017), `Higher Theory and the Three Problems of Physics', \textit{Preprint}, arXiv. \url{https://arxiv.org/pdf/1712.09454}} focusing on the relevance of topos theory in physics, higher topos theory, which captures precisely such transformations and mutations between several bodies, may embody structures that do not constitute an all-encompassing language of truth similar to logic even internally, as it were an affair `written into the thing itself' (cf. Wittgenstein). While the material conditions of this possibility fall outside the scope of this paper, I am safe to conclude that elementary topos theory---the first step towards displacing the question of materiality---does matter. To use Badiou's\footnote{Badiou, \LW, 2009, p. 309.}  own words, topos theory counts as an `expansive construction of the multiple to the multiple-beings'; in particular, the category of all topoi, $\Top$, is not itself a topos, in the same way as Russell demonstrated that the category of sets is not a set. The passage from localic to elementary and then higher  theory marks precisely this expansive moment, severely questioning any overarching attempts to arrest or fix this process according to a specific mode of (even relational) expression. Something prevents not just sets but also the theory of elementary topoi from completing an `existential seal' of mathematics---something challenging not just the factual description of being, but also the assumption that the one `who exist in the world'\footnote{Ibid., p. 310.} and bears the consequences of the event should be primarily singular. 

In conclusion, if Badiou defines the event negatively with respect to mathematical regulation, he remains unable to grasp how that regulation itself happens, for otherwise he would be forced to situate his own event, his own process of inquiry, alongside others, recognizing that his own formalism is but a historical exception. The categorical process instead falsifies at least one of Badiou's two primary arguments: (1) that the event cannot be formally approached or localized except by regulating its consequences; or (2) that formal logic is the only meta-structure, the only regulatory foundation available to the one who intervenes it. In other words, category theory either creates a form of expression that takes one beyond the domain of proper regulation, or one must expand how one frames the question of regulation itself. Any one of these two propositions can only survive upon the failure of the other.

Indeed, let me assume the second option. This reduces to saying that in a causal sense category theory does not genuinely exist: it fails to fulfill Badiou's test of `the true in the Alternative'---his \textit{second} definitive test of a body\footnote{Ibid., p. 430.}. This would render a topos, insofar as it resists set-theory, `poetic'; something external to the Platonic plane of causes, or the `plane of consistence' to borrow Gilles Deleuze and Félix Guattari\footnote{See Deleuze, Gilles \& Félix Guattari (1988), \textit{Thousand Plateaus---Capitalism and Schizophrenia}. Transl. Brian Massumi, Minneapolis, University of Minnesota Press. [Originally published in 1980.]}, which though would make it much closer to an event than what Badiou is willing to admit. Alternatively, one must reframe the idea of regulatory cause itself. This is to say that what happens (to mathematics) does not happen merely through a paradox.

Either way, what topos theory demonstrates is that what it regulates seems to change at least at the level of \textit{expression}---i.e., switching from a single, `generic' surrogate of truth ($\female$) to a whole variety of them, expressing choice as a process rather than an outcome. Whether or not this is consequential in a strictly causal sense might be a matter of taste; yet it certainly affects how the proofs and refutations are structured and inscribed, and especially their philosophical interpretation, as I will elaborate in the second article of this series\footnote{This is formally demonstrated in Veilahti, `Bridging Badiou's Gaps', 2026.}.

I wonder if anyone is willing to doubt the profound originality of Badiou's programmatic attempt to bridge the historical chasm between the philosophical and scientific discourses of materialism and causality. Not only as a philosopher, but also as a diligent, even if sometimes misleading student of mathematics, Badiou has opened Pandora's box, establishing an entirely new trajectory for contemporary science studies---something I have come to call \textit{mathematical empiricism}. He, inevitably, operates with a distinct style---a style characteristic of a `great writer' which, as Gilles Deleuze observes, `is always a style of life'\footnote{See  Deleuze, Gilles (1990), \textit{Negotiations, 1972--1990.} Trans. Martin Joughin. New York: Columbia University Press. p. 100.}.

While Badiou explicitly distances his framework from traditional epistemology, it is precisely this formal style that animates his singular perspective. Yet much remains to be said when that very style acquires new, non-local or category-theoretic modes of expression---a source of tranquility for the one destined to choose while facing anxious times ahead; one who views the `task
of language’ not primarily as `establish[ing] limits’, but in `go[ing] beyond
them'—or, as Deleuze\footnote{Deleuze, Gilles (2001), \textit{Logic of Sense}. Transl. C. V. Boundas. London, Athlone. [Originally published in 1969.]
p. 7.} alternatively says, `allow[ing] the active and passive to be interchanged more easily'. Inasmuch as (higher) topos theory dynamically embodies structural transformation between bodies, we require entirely new modalities to conceptualize not just `being' and subjectivity, but the very concepts of materiality, choice, and embodiment. 

It has been my vocation to revive that vital, yet formalist style of Badiou, but in a way sensitive to the changing subjectivities of mathematical expression. The complex history of this becoming of mathematics---that is, of the discourse of being---runs not merely through scientific schools and institutions; it remains visible across the entire canon of classical philosophy. Yet, in many respects, we inhabit an era of formal fragmentation, or rather, the multiplication of the communicative substances of language---an obvious challenge not just to science but also to democracy. While `dialectical materialism' 
is hardly the final answer, it certainly forms the first, singular step in this path to understanding real diversity.

 \newpage

\section*{References}

\setlength{\parindent}{0pt}

\setlength{\parskip}{1.5ex}
\small{

Awodey, S. (1996), `Structure in Mathematics and Logic: A Categorical Perspective'. \textit{Philosophia Mathematica} 4 (3). pp. 209--237. \\ doi: 10.1093/philmat/4.3.209.

Badiou, Alain (2006), \textit{Being and Event.} Transl. Oliver Feltham. London, New York: Continuum. [Originally published in 1988.]

Badiou, Alain (2006), \textit{Briefings on Existence: A Short Treatise on Transitory Ontology}, translated by Norman Madarasz, State University of New York Press. 

Badiou, Alain (2009), \textit{Logics of Worlds. Being and Event, 2.} Transl. Alberto Toscano. London and New York: Continuum. [Originally published in 2006.]

Badiou, Alain (2012), \textit{The Adventure of French Philosophy}. Trans. Bruno Bosteels. New York: Verso.

Badiou, Alain (2022), \textit{The Immanence of Truths: Being and Event III}, translated by A. Toscano and K. Reinhard, Bloomsbury Academic. 

Bell, J. L. (1988), \textit{Toposes and Local Set Theories: An Introduction.} Oxford: Oxford University Press. 

Borceux, Francis (1994), \textit{Handbook of Categorical Algebra. Basic Theory. Vol. I.} Cambridge: Cambridge University Press. 

Cohen, Paul J. (1963), `The Independence of the Continuum Hypothesis', \textit{Proc. Natl. Acad. Sci. USA} 50(6). pp. 1143--1148. 

Cohen, Paul J. (1964), `The Independence of the Continuum Hypothesis II', \textit{Proc. Natl. Acad. Sci. USA} 51(1). pp. 105--110.

Deleuze, Gilles (1990), \textit{Negotiations, 1972--1990.} Trans. Martin Joughin. New York: Columbia University Press.

Deleuze, Gilles (2001), \textit{Logic of Sense}. Transl. C. V. Boundas. London, Athlone. [Originally published in 1969.]

Deleuze, Gilles \& Félix Guattari (1988), \textit{Thousand Plateaus---Capitalism and Schizophrenia}. Transl. Brian Massumi, Minneapolis, University of Minnesota Press. [Originally published in 1980.]

Foucault, Michel (1995), \textit{Discipline and Punish: The Birth of the Prison}, trans. Alan Sheridan. New York: Random House.

Fourman, Michael P. \& Scott, Dana S. (1979), `Sheaves and logic', in M. P. Fourman, C. J. Mulvey \& D. S. Scott (eds.), \textit{Applications of Sheaves: Proceedings of the Research Symposium on Applications of Sheaf Theory to Logic, Algebra and Analysis, Durham, 1977}, Lecture Notes in Mathematics 753, Berlin: Springer, pp.~302--401.

Fraser, Mariam, Kember, Sarah \& Lury, Celia (2005), `Inventive Life. Approaches to the New Vitalism'. \textit{Theory, Culture \& Society} 22(1): 1--14.

Goldblatt, Robert (1984), \textit{The Categorical Analysis of Logic}. Mineola: Dover. 

Higgs, Denis (1973), `A category approach to Boolean-valued set theory', unpublished manuscript, University of Waterloo.

Johnstone, Peter T. (1977), \textit{Topos Theory}. London: Academic Press. 

Johnstone, Peter T. (2002), \textit{Sketches of an Elephant. A Topos Theory Compendium.} Volume 1. Oxford: Clarendon Press. 

Kant, Immanuel (1855), \textit{Critique of Pure Reason}. Trans. J. M. D. Meiklejohn. London: Henry G. Bohn.

Krömer, Ralf (2007), \textit{Tool and Object: A History and Philosophy of Category Theory}. Science Networks. Historical Studies 32. Berlin: Birkhäuser.

Mac Lane, Saunders \& Ieke Moerdijk (1992), \textit{Sheaves in Geometry and Logic. A First Introduction to Topos Theory}. New York: Springer-Verlag.

Madarasz, Norman (2005), `On Alain Badiou's Treatment of Category Theory in View of a Transitory Ontology', In Gabriel Riera (ed), \textit{Alain Badiou---Philosophy and its Conditions.} New York: State University of New York Press. 23--44.

Palmgren, E. (2009),  `Category Theory and Structuralism'. \\ url: 
\url{https://www2.math.uu.se/~palmgren/CTS-fulltext.pdf}, accessed Oct 21$\st$, 2026.

Shapiro, S. (1996), `Mathematical structuralism'. \textit{Philosophia Mathematica} 4(2), 81--82.

Shapiro, S. (2005) `Categories, structures, and the Frege-Hilbert controversy: The status of meta-mathematics'. \textit{Philosophia Mathematica} 13(1), 61--77.

Veilahti, Antti (2017), `Higher Theory and the Three Problems of Physics', \textit{Preprint}, arXiv. \url{https://arxiv.org/pdf/1712.09454}

Veilahti, Antti (2026), `Bridging Badiou's Gaps: Topos Theory and the Bodies of Mathematical Expression', \textit{Preprint}, arXiv. \url{https://arxiv.org/pdf/2611.27280}

Wyler, Oswald (1991), \textit{Lecture Notes on Topoi and Quasi-Topoi}. Singapore, New Jersey, London, Hong Kong: World Scientific. 

}

\end{document}